\documentclass[11pt]{article}
\usepackage[a4paper,height=24cm,width=16cm]{geometry}
\usepackage[english]{babel}
\usepackage{amsmath,amssymb,stmaryrd,mathrsfs,bbm,hyperref,xcolor,latexsym,theorem}

\catcode`\<=\active \def<{
\fontencoding{T1}\selectfont\symbol{60}\fontencoding{\encodingdefault}}
\newcommand{\assign}{:=}
\newcommand{\mathD}{\mathrm{D}}
\newcommand{\mathd}{\mathrm{d}}
\newcommand{\nobracket}{}
\newcommand{\nocomma}{}
\newcommand{\tmaffiliation}[1]{\\ #1}

\newcommand{\tmop}[1]{\ensuremath{\operatorname{#1}}}
\newcommand{\tmscript}[1]{\text{\scriptsize{$#1$}}}
\newcommand{\tmsep}{, }
\newcommand{\tmtextbf}[1]{{\bfseries{#1}}}
\newcommand{\tmtextit}[1]{{\itshape{#1}}}
\newenvironment{itemizedot}{\begin{itemize} }{\end{itemize}}
\newenvironment{proof}{\noindent\textbf{Proof\ }}{\hspace*{\fill}$\Box$\medskip}
\newenvironment{proof*}[1]{\noindent\textbf{#1\ }}{\hspace*{\fill}$\Box$\medskip}
\newtheorem{corollary}{Corollary}
\newtheorem{definition}{Definition}
\newtheorem{lemma}{Lemma}
\newtheorem{proposition}{Proposition}
{\theorembodyfont{\rmfamily}\newtheorem{remark}{Remark}}
\newtheorem{theorem}{Theorem}
\newcommand{\tmkeywords}{\textbf{Keywords:} }
\newcommand{\tmmsc}{\textbf{A.M.S. subject classification:} }

\newcommand{\VV}{\mathscr{C}}

\newcommand{\CF}{\mathscr{F}}
\newcommand{\CG}{\mathscr{G}}
\newcommand{\CS}{\mathscr{S}}

\begin{document}

\title{ The $\Phi^4_3$ measure via Girsanov's theorem}

\author{
  N.~Barashkov and M.Gubinelli
  \tmaffiliation{Hausdorff Center of Mathematics \&\\
  Institute of Applied Mathematics\\
  University of Bonn, Germany}\\
  \tt{\{barashkov,gubinelli\}@iam.uni-bonn.de}
}

\maketitle

\begin{abstract}
  We construct the $\Phi^4_3$ measure on a periodic three dimensional box as
  an absolutely continuous perturbation of a random shift of the Gaussian free
  field. The shifted measure is constructed via Girsanov's theorem and the
  relevant filtration is the one generated by a scale parameter. As a
  byproduct we give a self-contained proof that the $\Phi^4_3$ measure is
  singular wrt. the Gaussian free field.
\end{abstract}

\tmmsc{60H30}{\tmsep}{81T08}

\tmkeywords{Constructive Euclidean quantum field theory, Bou{\'e}--Dupuis
formula, paracontrolled calculus.}

\section{Introduction}

The $\Phi^4_3$ measure on the three dimensional torus \ $\Lambda
=\mathbbm{T}^3 = (\mathbbm{R}/ 2 \pi \mathbbm{Z})^3$ is the probability
measure $\nu$ on distributions $\CS' (\Lambda)$ corresponding to the formal
functional integral
\begin{equation}
  \nu (\mathd \varphi) = {}^{\backprime\backprime} \left\{ \frac{1}{Z} \exp
  \left[ - \lambda \int_{\Lambda} (\varphi^4 - \infty \varphi^2) \mathd x
  \right] \mu (\mathd \varphi) \right\}'' \label{eq:phi43}
\end{equation}
where $\mu$ is the law of the Gaussian free field with covariance $(1 -
\Delta)^{- 1}$ on $\Lambda$, $Z$ a normalization constant and $\lambda$ a
coupling constant. The $\infty$ appearing in this expression reminds us that
many things are wrong with this recipe. The key difficulty can be traced to
the fact that the measure we are looking for it is not absolutely continuous
wrt. the reference measure $\mu$. This fact seems part of the folklore even if
we could not find a rigorous proof for it in the available literature apart
from a work of Albeverio and Liang~{\cite{albeverio_remark_2008}} which
however refers to the Euclidean fields at time zero. The singularity of the
$\Phi^4_3$ measure is indeed a major technical difficulty in a rigorous study.
Obtaining a complete construction of this formal object (both in finite and
infinite volume) has been one of the main achievements of the constructive
quantum field theory
program~{\cite{glimm_positivity_1973,feldman_lambda_1974,park_lambda_1975,feldman_wightman_1976,magnen_infinite_1976,benfatto_ultraviolet_1980,brydges_new_1983}}.

\

In recent years the rigorous study of the $\Phi^4_3$ model has been pursued
from the point of view of \tmtextit{stochastic quantization}. In the original
formulation of Parisi--Wu~{\cite{parisi_perturbation_1981}}, stochastic
quantization is a way to introduce additional degrees of freedom (in
particular a dependence on a fictious time) in order to obtain an
\tmtextit{equation} whose solutions describe a measure of interest, in this
case the $\Phi^4_3$ measure on $\Lambda$ as in~(\ref{eq:phi43}) or its
counterpart in the full space. Rigorous analysis of stochastic quantization
for simpler models like $\Phi^4_2$ (the two-dimensional analog of
eq.~(\ref{eq:phi43})) \ started with the
work~{\cite{jona_lasinio_stochastic_1985}}. It has been only with the
fundamental work of Hairer on regularity
structures~{\cite{hairer_theory_2014}} that the three dimensional model could
be successfully attacked, see
also~{\cite{catellier_paracontrolled_2013,kupiainen_renormalization_2016}}.
This new perspective on this and related problems led to a series of new
results on the global space-time control of the stochastic
dynamics~{\cite{MWcomedown,gubinelli_global_2019,albeverio_invariant_2017,moinat_space_time_2018}}
and to a novel proof of the constructions of non-Gaussian Euclidean quantum
field theories in three dimensions~{\cite{gubinelli_pde_2018}}.

\

A conceptual advantage of stochastic quantization is that it is insensitive
to questions of absolute continuity wrt. to a reference measure. This, on the
other hand, is the main difficulty of the Gibbisan point of view as expressed
in eq.~(\ref{eq:phi43}). In order to explore further the tradeoffs of
different approaches we have recently developed a variational
method~{\cite{barashkov_gubinelli_variational}} for the construction and
\tmtextit{description} of $\Phi^4_3$ with which we were able to provide an
explicit formula for the Laplace transform of $\Phi^4_3$ in terms of a
stochastic control problem. In this control problem the controlled process
represents the scale-by-scale evolution of the interacting random field.

\

The present paper is the occasion to explore further this point of view by
constructing a novel measure via a random shift of the Gaussian free field and
proving that the $\Phi^4_3$ measure can be constructed as an absolutely
continuous perturbation thereof. Without entering into technical details now
let us give the broad outline of this construction. We consider a Brownian
martingale $(W_t)_{t \geqslant 0}$ with values in $\CS' (\Lambda)$ and such
that $W_t$ is a regularization of the Gaussian free field $\mu$ at scale $t$.
Let us denote $\mathbbm{P}$ its law, $\mathbbm{E}$ the corresponding
expectation. In particular $W_t \rightarrow W_{\infty}$ as $t \rightarrow
\infty$ and $W_{\infty}$ has law $\mu$. We can identify the $\Phi^4_3$ measure
$\nu$ as the weak limit $\nu^T \rightarrow \nu$ as $T \rightarrow \infty$ of
the family of probability measures $(\nu^T)_{T \geqslant 0}$ defined as
\[ \nu^T (A) =\mathbbm{P}^T (W_T \in A), \]
where $\mathbbm{P}^T$ is the measure on paths $(W_t)_{t \geqslant 0}$ with
density
\[ \frac{\mathd \mathbbm{P}^T}{\mathd \mathbbm{P}} = \frac{1}{Z_T} e^{- V_T
   (W_T)}, \]
and
\[ V_T (\varphi) \assign \lambda \int_{\Lambda} (\varphi (x)^4 - a_T \varphi
   (x)^2 + b_T) \mathd x, \]
is a quartic polynomial in the field $\varphi$ with $(a_T, b_T)_T$ a family of
(suitably diverging) renormalization constants. The presence of the scale
parameter $t \in \mathbbm{R}_+$ allows to introduce a filtration and a family
of measures $\mathbbm{Q}^v$ defined as the Girsanov transformation
\begin{equation}
  \left. \frac{\mathd \mathbbm{Q}^v}{\mathd \mathbbm{P}} \right|_{\CF_T} =
  \exp \left( L_T^v - \frac{1}{2} \langle L^v \rangle_T \right), \qquad L^v_t
  = \int_0^t \langle v_s, \mathd W_s \rangle_{L^2 (\Lambda)} \label{eq:girs}
\end{equation}
where $(\langle L^v \rangle_t)_{t \geqslant 0}$ is the quadratic variation of
the (scalar) local martingale $(L^v_t)_{t \geqslant 0}$ and $(v_t)_{t
\geqslant 0}$ is an adapted process with values in $L^2 (\Lambda)$. Let
\[ D_T \assign \frac{1}{Z_T} e^{- V_T (W_T)} \left( \frac{\mathd
   \mathbbm{Q}^v}{\mathd \mathbbm{P}} \right)^{- 1}, \]
be the density of $\mathbbm{P}^T$ wrt. $\mathbbm{Q}^v$. We will show that it
is possible to choose $v$ in such a way that the family $(D_T)_{T \geqslant
0}$ is uniformly integrable under $\mathbbm{Q}^v$ and that $D_T \rightarrow
D_{\infty}$ weakly in $L^1 (\mathbbm{Q}^v)$. With particular choice of $v$ we
call $\mathbbm{Q}^v$ the \tmtextit{drift} measure: it is the central object of
this paper. By Girsanov's theorem the canonical process $(W_t)_{t \geqslant
0}$ satisfies the equation
\[ \mathd W_t = v_t \mathd t + \mathd \tilde{W}_t, \qquad t \geqslant 0, \]
where $(\tilde{W}_t)_{t \geqslant 0}$ is a Gaussian martingale under
$\mathbbm{Q}^v$ (and has law equal to that of $(W_t)_{t \geqslant 0}$ under
$\mathbbm{P}$, that is is a regularized Gaussian free field). We will show
also that the drift $v_t$ can be written as a (polynomial) function of
$(\tilde{W}_s)_{s \in [0, t]}$, that is $v_t = \tilde{V}_t ((\tilde{W}_s)_{s
\in [0, t]})$. Therefore we have an explicit description of the process
$(W_t)_{t \geqslant 0}$ under the drift measure $\mathbbm{Q}^v$ as the unique
solution of the path-dependent SDE
\begin{equation}
  \mathd W_t = \tilde{V}_t ((\tilde{W}_s)_{s \in [0, t]}) \mathd t + \mathd
  \tilde{W}_t, \qquad t \geqslant 0. \label{eq:sde-girs}
\end{equation}
Let us note that this formula expresses the ``interacting'' random field
$(W_t)_t$ as a function of the ``free'' field $(\tilde{W}_t)_t$. In this
respect is a formula with very similar technical merits as the stochastic
quantization approach.

\

Intuitively this new measure $\mathbbm{Q}^v$, is half way between the
variational description in~{\cite{barashkov_gubinelli_variational}} and the
(formal) Gibbsian description of eq.~(\ref{eq:phi43}). It constitutes a
measure which is relatively explicit, easy to construct and analyze and which
can be used as reference measure for $\Phi^4_3$, very much like the Gaussian
free field can be used as reference measure for
$\Phi^4_2$~{\cite{glimm_quantum_1987}}.

\

As an application we provide a self-contained proof of the singularity of the
$\Phi^4_3$ measure $\nu$ wrt. the Gaussian free field $\mu$. \ As we already
remarked the singularity of $\Phi^4_3$ seems to belongs to the folklore and we
were not able to trace any written proof of that. However M.~Hairer, during a
conference at Imperial College in 2019 showed us an unpublished proof of him
of singularity using the stochastic quantization equation. Our proof and his
are very similar and we do not claim any essential novelty in this respect.
Albeit the proof is quite straightforward we wrote down all the details in
order to provide a reference for this fact. The main point of the present
paper remains that of describing the drift measure as a novel object in the
context of $\Phi^4_3$ and similar measures.

\

Our proof of singularity, in particular also shows that the drift measure
$\mathbbm{Q}^v$ is singular wrt. $\mathbbm{P}$. The intuitive reason is that
the drift $(V_t)_{t \geqslant 0}$ in the SDE~(\ref{eq:sde-girs}) is not
regular enough (as $t \rightarrow \infty$) to be along Cameron--Martin
directions for the law $\mathbbm{P}$ of the process $(W_t)_{t \geqslant 0}$
and therefore the Girsanov transform~(\ref{eq:girs}) gives a singular measure
when extended all the way to $T = + \infty$.

\

{\noindent}\tmtextbf{Acknowledgments.} M.G. would like to thank S. Albeverio,
D.~Brydges, C. Garban and M.~Hairer for interesting discussions on the topic
of singularity of $\Phi^4_3$. N.B would like to thank Bjoern Bringmann for
some helpful comments. The authors would like to thank the Isaac Newton
Institute for Mathematical Sciences for support and hospitality during the
program SRQ: Scaling limits, Rough paths, Quantum field theory during which
part of the work on this paper was undertaken. This work is supported by DFG
via CRC 1060 and by EPSRC via Grant Number EP/R014604/1.

\

{\noindent}\tmtextbf{Notations.} Let us fix some notations and objects.
\begin{itemizedot}
  \item For $a \in \mathbbm{R}^d$ we let $\langle a \rangle \assign (1 + | a
  |^2)^{1 / 2}$.
  
  \item The constant $\varepsilon > 0$ represents a small positive number
  which can be different from line to line.
  
  \item Denote with $\CS (\Lambda)$ the space of Schwartz functions on
  $\Lambda$ and with $\CS' (\Lambda)$ the dual space of tempered
  distributions. The notation $\hat{f}$ or $\CF f$ stands for the space
  Fourier transform of $f$ and we will write $g (\mathD)$ to denote the
  Fourier multiplier operator with symbol $g : \mathbbm{R}^n \rightarrow
  \mathbbm{R}$, i.e. $\CF (g (\mathD) f) = g \CF f$.
  
  \item $B^{\alpha}_{p, q} = B^{\alpha}_{p, q} (\Lambda)$ denotes the Besov
  spaces of regularity $\alpha$ and integrability indices $p, q$ as usual.
  $\VV^{\alpha} = \VV^{\alpha} (\Lambda)$ is the H{\"o}lder--Besov space
  $B^{\alpha}_{\infty, \infty}$, $W^{\alpha, p} = W^{\alpha, p} (\Lambda)$
  denote the standard fractional Sobolev spaces defined by the norm $\| f
  \|_{W^{s, q}} \assign \| \langle \mathD \rangle^s f \|_{L^q}$ and
  $H^{\alpha} = W^{\alpha, 2}$. The symbols $\prec, \succ, \circ$ denotes
  spatial paraproducts wrt. a standard Littlewood--Paley decomposition. The
  reader is referred to Appendix~\ref{sec:appendix-para} for an overview of
  the functional spaces and paraproducts.
\end{itemizedot}

\section{The setting}

The setting of this paper is the same of that of our variational
study~{\cite{barashkov_gubinelli_variational}}. In this section we will
briefly recall it and also state some results from that paper which will be
needed below. They concern the Bou{\'e}--Dupuis formula and certain estimates
which will be important also in our analysis of absolute continuity.

\

Let $\Omega \assign C \left( \mathbbm{R}_+ ; \VV^{- 3 / 2 - \varepsilon}
(\Lambda) \right)$ and $\CF$ be the Borel $\sigma$--algebra of $\Omega$. On
$\left( \Omega, \CF \right)$ consider the probability measure $\mathbbm{P}$
which makes the canonical process $(X_t)_{t \geqslant 0}$ a cylindrical
Brownian motion on $L^2 (\Lambda)$ and let $\left( \CF_t \right)_{t \geqslant
0}$ the associated filtration. In the following $\mathbbm{E}$ without any
qualifiers will denote expectations wrt. $\mathbbm{P}$ and
$\mathbbm{E}_{\mathbbm{Q}}$ will denote expectations wrt. some other measure
$\mathbbm{Q}$.

On the probability space $\left( \Omega, \CF, \mathbbm{P} \right)$ there
exists a collection $(B_t^n)_{n \in (\mathbbm{Z})^3}$ of complex
(2-dimensional) Brownian motions, such that $\overline{B^n_t} = B^{- n}_t$,
$B^n_t, B^m_t$ independent for $m \neq \pm n$ and $X_t = \sum_{n \in
\mathbbm{Z}^3} e^{i \langle n, \cdot \rangle} B^n_t$.

\

Fix some $\rho \in C_c^{\infty} (\mathbbm{R}_+, \mathbbm{R}_+)$,decreasing,
such that $\rho = 1$ on $B (0, 9 / 10)$ and $\tmop{supp} \rho \subset B (0,
1)$. For $x \in \mathbbm{R}^3$ let $\rho_t (x) \assign \rho (\langle x \rangle
/ t)$ with $\langle x \rangle \assign (1 + | x |^2)^{1 / 2}$ and
\[ \sigma_t (x) \assign \left( \frac{\mathd}{\mathd t} (\rho^2_t (x))
   \right)^{1 / 2} = (- 2 (\langle x \rangle / t) \rho (\langle x \rangle / t)
   \rho' (\langle x \rangle / t))^{1 / 2} / t^{1 / 2} . \]
Let $J_s = \sigma_s (\mathD) \langle \mathD \rangle^{- 1}$ and consider the
process $(W_t)_{t \geqslant 0}$ defined by
\begin{equation}
  W_t \assign \int_0^t J_s \mathd X_s = \sum_{n \in \mathbbm{Z}^3} e^{i
  \langle n, \cdot \rangle} \int_0^t \frac{\sigma_s (n)}{\langle n \rangle} d
  B^n_s, \qquad t \geqslant 0. \label{eq:def-Y}
\end{equation}
It is a centered Gaussian process with covariance
\begin{eqnarray*}
  \mathbbm{E} [\langle W_t, \varphi \rangle \langle W_s, \psi \rangle] & = &
  \sum_{n \in \mathbbm{Z}^3} \frac{\rho_{\min (s, t)}^2 (n)}{\langle n
  \rangle^2} \hat{\varphi} (n) \overline{\hat{\psi} (n)},
\end{eqnarray*}
for any $\varphi, \psi \in \CS (\Lambda)$ and $t, s \geqslant 0$, by Fubini
theorem and Ito isometry. By dominated convergence $\lim_{t \rightarrow
\infty} \mathbbm{E} [\langle W_t, \varphi \rangle \langle W_t, \psi \rangle] =
\sum_{n \in \mathbbm{Z}^3} \langle n \rangle^{- 2} \hat{\varphi} (n)
\overline{\hat{\psi} (n)}$ for any $\varphi \nocomma, \psi \in L^2 (\Lambda)$.
For any finite ``time'' $T$ the random field $W_T$ on $\Lambda$ has a bounded
spectral support and the stopped process $W^T_t = W_{t \wedge T}$ for any
fixed $T > 0$, is in $C (\mathbbm{R}_+, C^{\infty} (\Lambda))$. Furthermore
$(W^T_t)_t$ only depends on a finite subset of the Brownian motions $(B^n)_{n
\in \mathbbm{Z}^3}$. We write $g (\mathD)$ for the Fourier multiplier operator
with symbol $g$.

\

Observe that $J_t$ has the property, that for some function $f \in B_{p,
p}^s$ or $f \in W_{}^{s, p}$ with $p \in [1, \infty]$ and $s \in \mathbbm{R}$,
for any $\alpha \in \mathbbm{R}$
\[ \| J_t f \|_{B_{p, p}^{s + 1 - \alpha}} \lesssim \langle t \rangle^{-
   \alpha - 1 / 2} \| f \|_{B_{p, p}^s} . \]

We will denote by $\llbracket W^n_t \rrbracket$, $n = 1, 2, 3$, the $n$-th
Wick-power of the Gaussian random variable $W_t$ (under $\mathbbm{P}$) and
introduce the convenient notations $\mathbbm{W}_t^2 \assign 12 \llbracket
W^2_t \rrbracket$, $\mathbbm{W}_t^3 \assign 4 \llbracket W_t^3 \rrbracket$.
Furthermore we will write $\llbracket (\langle \mathD \rangle^{- 1 / 2} W_t)^n
\rrbracket, n \in \mathbbm{N}$ for the $n$-th Wick-power of $\langle \mathD
\rangle^{- 1 / 2} W_t$ . It exists for any $0 < t < \infty$ and any $n
\geqslant 1$ since it is easy to see that $\langle \mathD \rangle^{- 1 / 2}
W_t$ has a covariance with a diagonal behavior which can be controlled by
$\log \langle t \rangle$. These Wick powers converge as $T \rightarrow \infty$
in spaces of distributions with regularities given in the following table:

\begin{table}[h]
\centering
  \begin{tabular}{cccc}
    $W$ & $\mathbbm{W}^2$ & $s \mapsto J_s \mathbbm{W}_s^3$ &
    $\llbracket (\langle \mathD \rangle^{- 1 / 2} W)^n \rrbracket$\\
    \hline
    $C \VV^{- 1 / 2 -}$ & $C \VV^{- 1 -}$ & $C \VV^{- 1 / 2 -} \cap L^2 \VV^{-
    1 / 2 -}$ & $C \VV^{0 -}$
  \end{tabular}
  \caption{\label{table:reg}Regularities of the various stochastic objects,
  the domain of the time variable is understood to be $[0, \infty]$, $C
  \VV^{\alpha} = C \left( [0, \infty] ; \VV^{\alpha} \right)$ and $L^2
  \VV^{\alpha} = L^2 \left( \mathbbm{R}_+ ; \VV^{\alpha} \right)$. Estimates
  in these norms holds a.s. and in $L^p (\mathbbm{P})$ for all $p \geqslant 1$
  (see {\cite{barashkov_gubinelli_variational}}).}
\end{table}

We denote by $\mathbbm{H}_a$ the space of $\left( \CF_t \right)_{t \geqslant
0}$-progressively measurable processes which are $\mathbbm{P}$-almost surely
in $\mathcal{H} \assign L^2 (\mathbbm{R}_+ \times \Lambda)$. We say that an
element $v$ of $\mathbbm{H}_a$ is a \tmtextit{drift}. Below we will need also
drifts belonging to $\mathcal{H}^{\alpha} \assign L^2 (\mathbbm{R}_+ ;
H^{\alpha} (\Lambda))$ for some $\alpha \in \mathbbm{R}$ where $H^{\alpha}
(\Lambda)$ is the Sobolev space of regularity $\alpha$. We denote the
corresponding space with $\mathbbm{H}_a^{\alpha}$. For any $v \in
\mathbbm{H}_a$ define the measure $\mathbbm{Q}^v$ on $\Omega$ by
\[ \frac{\mathd \mathbbm{Q}^v}{\mathd \mathbbm{P}} = \exp \left[
   \int_0^{\infty} v_s \mathd X_s - \frac{1}{2} \int_0^{\infty} \| v_s \|^2
   \mathd s \right] . \]
Denote with $\mathbbm{H}_c \subseteq \mathbbm{H}_a$ the set of drifts $v \in
\mathbbm{H}_a$ for which $\mathbbm{Q}^v (\Omega) = 1$, and set $W^v \assign W
- I (v)$, where
\[ I_t (v) = \int^t_0 J_s v_s \mathd s. \]
Below will need also the following operators. For all $t \geqslant 0$ let
$\theta_t : \mathbbm{R}^3 \rightarrow [0, 1]$ be a smooth function such that
\begin{equation}
  \begin{array}{lll}
    \theta_t (\xi) \sigma_s (\xi) & = & \text{$0$ for $s \geqslant t$,}\\
    \theta_t (\xi) & = & \text{$1$ for $| \xi | \leqslant t / 2$ \ provided
    that $t \geqslant T_0$}
  \end{array} \label{theta}
\end{equation}
for some $T_0 > 0$. For example one can fix smooth functions $\tilde{\theta},
\eta : \mathbbm{R}^3 \rightarrow \mathbbm{R}_+$ such that $\tilde{\theta}
(\xi) = 1$ if $| \xi | \leqslant 1 / 2$ and $\tilde{\theta} (\xi) = 0$ if $|
\xi | \geqslant 2 / 3$ , $\eta (\xi) = 1$ if $| \xi | \leqslant 1$ and $\eta
(\xi) = 0$ if $| \xi | \geqslant 2$. Then let $\tilde{\theta}_t (\xi) \assign
\tilde{\theta} (\xi / t)$ and define
\[ \theta_t (\xi) = (1 - \eta (\xi)) \tilde{\theta}_t (\xi) + \zeta (t) \eta
   (\xi) \tilde{\theta}_t (\xi) \]
where $\zeta (t) : \mathbbm{R}_+ \rightarrow \mathbbm{R}$ is a smooth function
such that $\zeta (t) = 0$ for $t \leqslant 10$ and $\zeta (t) = 1$ for $t
\geqslant 3$. Then eq~(\ref{theta}) will hold with $T_0 = 3$. We will let
\begin{equation}
  f^{\flat} \assign \theta (\mathD) f \label{eq:flat}
\end{equation}
for any $f \in \CS' (\Lambda)$..

\

Our aim here to study the measures $\mu_T$ defined on $\VV^{- 1 / 2 -
\varepsilon}$ as
\[ \frac{\mathd \mu_T}{\mathd \mathbbm{P}} = e^{- V_T (W_T)} \]
for $\varphi \in C^{\infty} (\Lambda)$
\begin{equation}
  V_T (\varphi) \assign \lambda \int_{\Lambda} (\varphi^4 - a_T \varphi^2 +
  b_T) \mathd x, \label{eq:def-V}
\end{equation}
with suitable $a_T, b_T \rightarrow \infty$. For convenience the measure
$\mu^T$ is not normalized and, wrt. to the notations in the introduction we
have
\[ \frac{\mathd \mathbbm{P}^T}{\mathd \mu^T} = \frac{1}{\mu^T (\Omega)} . \]

With these notations we can recall the following results of
{\cite{barashkov_gubinelli_variational}}.

\begin{theorem}
  \label{eq:th1}For any $a_T, b_T \in \mathbbm{R}$, and \ $f : \VV^{- 1 / 2 -
  \varepsilon} (\Lambda) \rightarrow \mathbbm{R}$ with linear growth,
  recall~(\ref{eq:def-V}) and let $V^f_T (\varphi) \assign f (\varphi) + V_T
  (\varphi) .$ Then the formula
  \begin{equation}
    \int_{\CS' (\Lambda)} e^{- V^f_T (\varphi)} \mu (\mathd \varphi)^{} = -
    \log \mathbbm{E} [e^{- V^f_T (W_T)}] = \inf_{u \in \mathbbm{H}_a}
    \mathbbm{E} \left[ V^f_T (W_T + I_T (u)) + \frac{1}{2} \int^T_0 \| u_t
    \|^2_{L^2 (\Lambda)} \mathd t \right] \label{eq:contr-prob}
  \end{equation}
  holds for any finite $T$.
\end{theorem}

This is a consequence of the more general Bou{\'e}--Dupuis formula which can
be stated as follows.

\begin{theorem}[BD formula]
  Assume $F : C ([0, T], C^{\infty} (\Lambda)) \rightarrow \mathbbm{R}$, be
  Borel measurable and such that there exist $p, q \in (1, \infty)$, with $1 /
  p + 1 / q = 1$, $\mathbbm{E} [| F (W) |^p] < \infty$ and $\mathbbm{E} [|
  e^{- F (W)} |^q] < \infty$(where we can regard $W$ as an element of $C ([0,
  T], C^{\infty} (\Lambda))$ by restricting to $[0, T]$). Then
  \begin{equation}
    - \log \mathbbm{E} [e^{- F (W)}] = \inf_{u \in \mathbbm{H}_a} \mathbbm{E}
    \left[ F (W + I (u)) + \frac{1}{2} \int^T_0 \| u_s \|_{L^2 (\Lambda)}^2
    \right] . \label{eq:bd}
  \end{equation}
\end{theorem}

We will use several times below the Bou{\'e}--Dupuis formula~(\ref{eq:bd}) in
order to control exponential integrability of various functionals. By a
suitable choice of renormalization and a change of variables in the control
problem~(\ref{eq:contr-prob}) we were able
in~{\cite{barashkov_gubinelli_variational}} to control the functional in
Theorem~\ref{eq:th1} uniformly up to infinity:

\begin{theorem}
  \label{thm:equicoerc}There exist a sequence $(a_T, b_T)_T$ with $a_T, b_T
  \rightarrow \infty$ as $T \rightarrow \infty$, such that
  \begin{eqnarray*}
    &  & \mathbbm{E} \left[ V^f_T (W_T + I_T (u)) + \frac{1}{2} \int^T_0 \|
    u_t \|_{L^2 (\Lambda)} \mathd t \right]\\
    & = & \mathbbm{E} \left[ \Psi^f_T (W, I (u)) + \lambda \int (I_T (u))^4 +
    \frac{1}{2} \| l^T (u) \|_{\mathcal{H}}^2 \right]
  \end{eqnarray*}
  where (Recall that $I_t^{\flat} (u) = \theta (\mathD) I_t (u)$
  by~(\ref{eq:flat}))
  \begin{equation}
    l_t^T (u) \assign u_t + \lambda \mathbbm{1}_{t \leqslant T}
    \mathbbm{W}^{\langle 3 \rangle}_t + \lambda \mathbbm{1}_{t \leqslant T}
    J_t (\mathbbm{W}_t^2 \succ I_t^{\flat} (u)) \label{eq:renorm-drift}
  \end{equation}
  and the functionals $\Psi^f_T : C ([0, T], C^{\infty} (\Lambda)) \times C
  ([0, T], C^{\infty} (\Lambda)) \rightarrow \mathbbm{R}$ satisfy the
  following bound
  \[ | \Psi^f_T (W, I (u)) | \leqslant Q_T (W) + \frac{1}{4} (\| I_T (u)
     \|^4_{L^4} + \| l^T (u) \|_{\mathcal{H}}^2) \]
  where $Q_T (W)$ is a function of $W$ independent of $u$ and such that
  $\sup_T \mathbbm{E} [| Q_T (W) |] < \infty$.
\end{theorem}

As a consequence we obtain the following corollary(See corollary 1 and Lemma 6
in {\cite{barashkov_gubinelli_variational}})

\begin{corollary}
  \label{cor:tightness}For $f : \VV^{- 1 / 2 - \varepsilon} (\Lambda)
  \rightarrow \mathbbm{R}$ with linear the bound:
  \[ - C \leqslant \mathbbm{E}_{\mu_T} [e^f] \leqslant C \]
  holds, with a constant $C$ independent of $T$. In particular $\mu_T$ is
  tight on $\VV^{- 1 / 2 - \varepsilon}$.
\end{corollary}

\

\section{Construction of the drift measure}

We start now to implement the strategy discussed in the introduction:
construct a shifted measure sufficiently similar to $\Phi^4_3$. Intuitively
the $\Phi^4_3$ measure should give rise to a canonical process which is a
shift of the Gaussian free field with a drift of the form given by
eq.~(\ref{eq:renorm-drift}). Indeed this drift $u$ should be the optimal drift
in the variational formula.

\

A small twist is given by the fact that the relevant Gaussian free field
entering these considerations is not the process $W = W (X)$ but that obtained
from the shifted canonical process $X^u_t = X_t - \int_0^t u_s \mathd s$ which
we denote by
\[ W^u \assign W (X^u) = W - I (u) . \]
Moreover for technical reasons we have to modify the drift in large scales and
add some coercive term which will allow later to prove some useful estimates.
We define the functional
\begin{equation}
  \Xi_s (W, u) \assign - \lambda J_s \mathbbm{W}_s^3 - \lambda \mathbbm{1}_{\{
  s \geqslant \bar{T} \}} J_s (\mathbbm{W}_s^2 \succ I_s^{\flat} (u)) - J_s
  \langle \mathD \rangle^{- 1 / 2} (\llbracket (\langle \mathD \rangle^{- 1 /
  2} W_s)^n \rrbracket), \qquad s \geqslant 0, \label{eq:pre-pre-drift}
\end{equation}
where $\bar{T} > 0, n \in \mathbbm{N}$ are constants which will be fixed later
and where we understand all the Wick renormalizations to be given functions of
$W$. We look now for the solution $u$ of the equation
\begin{equation}
  u = \Xi (W^u, u) = \Xi (W - I (u), u) . \label{eq:pre-drift}
\end{equation}
Expanding the Wick polynomials appearing in $\Xi (W - I (u), u)$ we obtain the
equation
\begin{equation}
  \begin{array}{lll}
    u_s & = & \Xi (W - I (u), u)\\
    & = & - \lambda J_s [\mathbbm{W}_s^3 -\mathbbm{W}_s^2 I_s (u) + 12 W_s
    (I_s (u))^2 - 4 (I_s (u))^3]\\
    &  & - \lambda \mathbbm{1}_{\{ s \geqslant \bar{T} \}} J_s
    [((\mathbbm{W}_s^2 - 24 W_s I_s (u) + 12 (I_s (u))^2)) \succ I_s^{\flat}
    (u)]\\
    &  & - \sum_{i = 0}^n \binom{n}{i} J_s \langle \mathD \rangle^{- 1 / 2}
    [\llbracket (\langle \mathD \rangle^{- 1 / 2} W_s)^i \rrbracket (- \langle
    \mathD \rangle^{- 1 / 2} I_s (u))^{n - i}]
  \end{array} \label{eq:drift}
\end{equation}
for all $s \geqslant 0$. This is an integral equation for $t \mapsto u_t$ with
smooth coefficients depending smoothly on $W$ and can be solved via standard
methods. Since the coefficients are of polynomial growth we must expect
explosion in finite time, so we have to be careful. Note that for any finite
time the process $(u_s)_{s > 0}$ has bounded spectral support. As a
consequence we can solve the equation in $L^2$ and as long as $\int^t_0 \| u
\|^2_{L^2} \mathd s$ is finite we can see from the equation that $\sup_{s
\leqslant t} \| u_s \|^2_{L^2}$ is finite. Therefore by the existence of local
solutions we have that, for all $N \geqslant 0$, the stopping time
\[ \tau_N \assign \inf \left\{ t \geqslant 0 \middle| \int^t_0 \| u \|^2_{L^2}
   \mathd s \geqslant N \right\}, \]
is strictly positive $\mathbbm{P}$-almost surely and $u$ exists up to the
(explosion) time $T_{\exp} \assign \sup_{N \in \mathbbm{N}} \tau_N$. Moreover,
by construction, the process $u_t^N \assign \mathbbm{1}_{\{ t \leqslant \tau_N
\}} u_t$ satisfies Novikov's condition, so it is in $\mathbbm{H}_c$ and by
Girsanov transformation we can define the probability measure on $C \left(
\mathbbm{R}_+, \VV^{- 1 / 2 - \varepsilon} (\Lambda) \right)$ given by
\[ \mathd \mathbbm{Q}^{u^N} \assign e^{\int^{\infty}_0 u_s^N \mathd X_s -
   \frac{1}{2} \int^{\infty}_0 \| u_s^N \|^2_{L^2 (\Lambda)} \mathd s} \mathd
   \mathbbm{P}, \]
and under which $X^{u^N}_t = X_t - \int^t_0 u_s^N \mathd s$ is a cylindrical
Brownian motion. In particular $(W^{u^N}_t)_{t \geqslant 0}$, given by
$\int^t_0 J_s \mathd X^{u^N}_s$ has under $\mathbbm{Q}^{u^N}$ the same law as
$(W_t)_{t \geqslant 0}$ has under $\mathbbm{P}$. Moreover we have that
$W^{u^N}_s = W^u_s$ for $0 \leqslant s \leqslant \tau_N$ and that $u$
satisfies the equation
\begin{equation}
  u_s = - \lambda J_s \mathbbm{W}_s^{u, 3} - \lambda \mathbbm{1}_{\{ s
  \geqslant \bar{T} \}} J_s (\mathbbm{W}_s^{u, 2} \succ I_t^{\flat} (u)) - J_s
  \langle \mathD \rangle^{- 1 / 2} (\llbracket (\langle \mathD \rangle^{- 1 /
  2} W^u_s)^n \rrbracket), \label{eq:drift-uu} \qquad s \in [0, \tau_N],
\end{equation}
where we introduced the notations $\mathbbm{W}_s^{u, 3} \assign 4 \llbracket
(W_s^u)^3 \rrbracket$ and $\mathbbm{W}_s^{u, 2} \assign 12 \llbracket
(W^u_s)^2 \rrbracket$.

\

Note that here the Wick powers are still taken to be given functions of $W$,
i.e we are still taking the Wick ordering with respect to the law of $W$ under
$\mathbbm{P}$ (or the law of $W^{u^N}$ under $\mathbbm{Q}^{u^N}$).

\

If we think of the terms containing $W^u$ as given (that is, we ignore their
dependence on $u$), eq.~(\ref{eq:drift-uu}) is a linear integral equation in
$u$ which can be estimated via Gronwall-type arguments. In order to do so, let
us denote by $U : H \mapsto \hat{u}$ the solution map of the equation
\begin{equation}
  \hat{u} = \Xi (H, \hat{u}) . \label{eq:map-U}
\end{equation}
This last equation is linear and therefore has nice global solutions (let's
say in $C (\mathbbm{R}_+, L^2)$) and by uniqueness and eq.~(\ref{eq:drift-uu})
we have $u_t = U_t (W^u)$ for $t \in [0, T_{\exp})$. From this perspective the
residual dependence on $u$ will not play any role since under the shifted
measure the law of the process $W^u$ does not depend on $u$. By standard
paraproduct estimates we have
\begin{eqnarray*}
  \| I_t (u) \|_{L^{\infty}} & \lesssim & \tilde{H}_t + \int^t_0
  \mathbbm{1}_{\{ s \geqslant \bar{T} \}} \| J^2_s (\mathbbm{W}_s^{u, 2} \succ
  I_s^{\flat} (u)) \|_{L^{\infty}} \mathd s\\
  & \lesssim & \tilde{H}_t +_{} \bar{T}^{- \varepsilon} \int^t_0 \langle s
  \rangle^{- 3 / 2} \| \mathbbm{W}_s^{u, 2} \|_{\VV^{- 1 - \varepsilon}} \|
  I_s^{\flat} (u) \|_{L^{\infty}} \mathd s,
\end{eqnarray*}
where we have used the presence of the cutoff $\mathbbm{1}_{\{ s \geqslant
\bar{T} \}}$ to introduce the small factor $T^{- \varepsilon}$ and we have
employed the notation
\[ \tilde{H}_t = \int^t_0 [\|J^2_s \mathbbm{W}^{u, 3}_s \|_{L^{\infty}} +
   \| J_s \langle \mathD \rangle^{- 1 / 2} (\llbracket (\langle \mathD
   \rangle^{- 1 / 2} W^u_s)^n \rrbracket) \|_{L^{\infty}}] \mathd s \]
\[ \lesssim \int^t_0 \frac{1}{\langle s \rangle^{1 / 2 - \varepsilon}} \|J_s
   \mathbbm{W}^{u, 3}_s \|_{\VV^{- 1 / 2 - \varepsilon}} \mathd s +
   \int^t_0 \frac{1}{\langle s \rangle^{3 / 2}} \| \llbracket (\langle \mathD
   \rangle^{- 1 / 2} W^u_s)^n \rrbracket \|_{H^{- 1 / 2}} \mathd s. \]
Therefore, by Gronwall's lemma
\begin{equation}
  \begin{array}{lll}
    \sup_{t \leqslant \tau_N} \| I_t (u) \|_{L^{\infty}} & \lesssim &
    \tilde{H}_{\tau_N} \exp \left( C \bar{T}^{- \varepsilon} \int^{\tau_N}_0
    \|\mathbbm{W}_s^{u, 2} \|_{\VV^{- 1 - \varepsilon}} \frac{\mathd
    s}{\langle s \rangle^{1 + \varepsilon}} \right) .
  \end{array} \label{boundIu}
\end{equation}
Under $\mathbbm{Q}^{u^N}$, the terms in $\tilde{H}_{\tau_N}$ are in all the
$L^p$ spaces by hypercontractivity and moreover for any $p \geqslant 1$ one
can choose $\bar{T}$ large enough to that also the exponential term is in
$L^p$. Using eq.~(\ref{eq:drift-uu}) it is then not difficult to show that
$\mathbbm{E}_{\mathbbm{Q}^{u^{N_1}}} [\| u^{N_2} \|^p_{\mathcal{H}^{- 1 / 2 -
\varepsilon}}] < \infty$ for any $p > 1$ (again provided we take $\bar{T}$
large enough depending on $p$) as long as $N_1 > N_2$. By the spectral
properties of $J$ and the equation for $u$, the process $t \mapsto
\mathbbm{1}_{\{ t \leqslant T \}} u_t$ is spectrally supported in a ball of
radius $T$, so we get in particular that
\[ \mathbbm{E}_{\mathbbm{Q}^{u^{N_1}}} \left[ \int^{\tau_{N_2} \wedge T}_0 \|
   u_s \|^2_{L^2} \mathd s \right] \lesssim T^{1 + \varepsilon}, \]
uniformly for any choice of $N_1 \geqslant N_2 \geqslant 0$.

\begin{lemma}
  \label{Lemma:changeofvariables}The family $(\mathbbm{Q}^{u^N})_N$ weakly
  converges to a limit $\mathbbm{Q}^u$ on $C \left( \mathbbm{R}_+, \VV^{- 3 /
  2 - \varepsilon} \right)$. Under $\mathbbm{Q}^u$ it holds $T_{\exp} =
  \infty$ almost surely and $\tmop{Law}_{\mathbbm{Q}^u} (X^u) =
  \tmop{Law}_{\mathbbm{P}} (X)$. Moreover for any finite $T$
  \[ \frac{\mathd \mathbbm{Q}^u |_{\CF_T}}{\mathd \mathbbm{P}|_{\CF_T}} = \exp
     \left( \int^T_0 u_s \mathd X_s - \frac{1}{2} \int^T_0 \| u_s \|^2_{L^2}
     \mathd s \right) . \]
\end{lemma}

\begin{proof}
  Consider the filtration $\left( \CG_N = \CF_{\tau_N} \right)_N$ and observe
  that $\left( \nobracket \mathbbm{Q}^{u^N} |_{\CG_N} \right)_N$ is a
  consistent family of inner regular probability distributions and therefore
  there exists a unique extension $\mathbbm{Q}^u$ to $\CG_{\infty} = \vee_N
  \CG_N$. Next observe that $\{ T_{\exp} < \infty \} = \bigcup_{T \in
  \mathbbm{N}} \{ T_{\exp} < T \} \subset \bigcup_{T \in \mathbbm{N}}
  \bigcap_{N \in \mathbbm{N}} \{ \tau_N < T \}$ and that for any $N, T <
  \infty$, we have
  \[ \mathbbm{E}_{\mathbbm{Q}^u} \left[ \int^{\tau_N \wedge T}_0 \| u_s
     \|^2_{L^2} \mathd s \right] =\mathbbm{E}_{\mathbbm{Q}^{u^N}} \left[
     \int^{\tau_N \wedge T}_0 \| u_s \|^2_{L^2} \mathd s \right] \lesssim T^{1
     + \varepsilon} . \]
  On the event $\{ \tau_N \leqslant T \}$ we have
  \[ \int^{\tau_N \wedge T}_0 \| u_s \|^2_{L^2} \mathd s = N, \]
  and therefore we also have \ $\mathbbm{Q}^u (\{ \tau_N \leqslant T \})
  \leqslant C T^{1 + \varepsilon} N^{- 1}$ which in turn implies
  $\mathbbm{Q}^u (T_{\exp} < T) = 0$. This proves that $T_{\exp} = + \infty$
  under $\mathbbm{Q}^u$, almost surely. As a consequence we can extend
  $\mathbbm{Q}^u$ to all of $\CF = \vee_T \CF_T$ since for any $A \in \CF_T$
  we can set
  \[ \mathbbm{Q}^u (A) =\mathbbm{Q}^u (A \cap \{ T_{\exp} = + \infty \}) =
     \lim_N \mathbbm{Q}^u (A \cap \{ T_{\exp} = + \infty, \tau_N \geqslant T
     \}) = \lim_N \mathbbm{Q}^{u^N} (A \cap \{ \tau_N \geqslant T \}) . \]
  If $A \in \CF_T$ then
  \[ \begin{array}{lll}
       \mathbbm{E}_{\mathbbm{Q}^u} [\mathbbm{1}_A (X^u)] & = & \lim_{N
       \rightarrow \infty} \mathbbm{E}_{\mathbbm{Q}^u} [\mathbbm{1}_{A \cap \{
       T \leqslant \tau_N \}} (X^u)] = \lim_{N \rightarrow \infty}
       \mathbbm{E}_{\mathbbm{Q}^{u^N}} [\mathbbm{1}_{A \cap \{ T \leqslant
       \tau_N \}} (X^{u^N})]\\
       & = & \lim_{N \rightarrow \infty} \mathbbm{E}_{\mathbbm{P}}
       [\mathbbm{1}_{A \cap \{ T \leqslant \tau_N \}} (X)]
     \end{array} \]
  and
  \[ \lim_{N \rightarrow \infty} \mathbbm{E}_{\mathbbm{P}} [\mathbbm{1}_{\{ T
     > \tau_N \}} (X)] = \lim_{N \rightarrow \infty}
     \mathbbm{E}_{\mathbbm{Q}^{u^N}} [\mathbbm{1}_{\{ T > \tau_N \}}
     (X^{u^N})] = \lim_{N \rightarrow \infty} \mathbbm{E}_{\mathbbm{Q}^u}
     [\mathbbm{1}_{\{ T > \tau_N \}} (X^u)] \rightarrow 0. \]
  This establishes that $\tmop{Law}_{\mathbbm{Q}^u} (X^u) =
  \tmop{Law}_{\mathbbm{P}} (X)$. On the other hand if $A \in \CF_T$ we have,
  using the martingale property of the Girsanov density,
  \[ \mathbbm{E}_{\mathbbm{Q}^u} [\mathbbm{1}_A] = \lim_{N \rightarrow \infty}
     \mathbbm{E}_{\mathbbm{Q}^u} [\mathbbm{1}_{A \cap \{ T \leqslant \tau_N
     \}}] = \lim_{N \rightarrow \infty} \mathbbm{E}_{\mathbbm{Q}^{u^N}}
     [\mathbbm{1}_{A \cap \{ T \leqslant \tau_N \}}] \]
  \[ = \lim_{N \rightarrow \infty} \mathbbm{E} \left[ \mathbbm{1}_{A \cap \{ T
     \leqslant \tau_N \}} e^{\int^{\tau_N}_0 u_s \mathd X_s - \frac{1}{2}
     \int^{\tau_N}_0 \| u_s \|^2_{L^2} \mathd s} \right] \]
  \[ = \lim_{N \rightarrow \infty} \mathbbm{E} \left[ \mathbbm{1}_{A \cap \{ T
     \leqslant \tau_N \}} e^{\int^T_0 u_s \mathd X_s - \frac{1}{2} \int^T_0 \|
     u_s \|^2_{L^2} \mathd s} \right] . \]
  And also
  \[ \lim_{N \rightarrow \infty} \mathbbm{E} \left[ \mathbbm{1}_{A \cap \{ T >
     \tau_N \}} e^{\int^T_0 u_s \mathd X_s - \frac{1}{2} \int^T_0 \| u_s
     \|^2_{L^2} \mathd s} \right] \leqslant \lim_{N \rightarrow \infty}
     \mathbbm{E} \left[ \mathbbm{1}_{\{ T > \tau_N \}} e^{\int^T_0 u_s \mathd
     X_s - \frac{1}{2} \int^T_0 \| u_s \|^2_{L^2} \mathd s} \right] \]
  \[ = \lim_{N \rightarrow \infty} \mathbbm{E} \left[ \mathbbm{1}_{\{ T >
     \tau_N \}} e^{\int^{\tau_N}_0 u_s \mathd X_s - \frac{1}{2}
     \int^{\tau_N}_0 \| u_s \|^2_{L^2} \mathd s} \right] = \lim_{N \rightarrow
     \infty} \mathbbm{E}_{\mathbbm{Q}^{u^N}} [\mathbbm{1}_{\{ T > \tau_N \}}]
     = 0. \]
  As a consequence
  \[ \mathbbm{E}_{\mathbbm{Q}^u} [\mathbbm{1}_A] =\mathbbm{E} \left[
     \mathbbm{1}_A e^{\int^T_0 u_s \mathd X_s - \frac{1}{2} \int^T_0 \| u_s
     \|^2_{L^2} \mathd s} \right] \]
  and therefore
  \[ \frac{\mathd \mathbbm{Q}^u |_{\CF_T}}{\mathd \mathbbm{P}|_{\CF_T}} =
     e^{\int^T_0 u_s \mathd X_s - \frac{1}{2} \int^T_0 \| u_s \|^2_{L^2}
     \mathd s}, \]
  as claimed.
\end{proof}

The following lemma will also be useful in the sequel and it is a consequence
of the above discussion:

\begin{lemma}
  For any $p > 1$ there exists a suitable choice of $\bar{T}$ such that
  \[ \mathbbm{E}_{\mathbbm{Q}^u} [\sup_{t \geqslant 0} \| I_t (u)
     \|^p_{L^{\infty}}] < \infty . \]
\end{lemma}

\begin{proof}
  This follows from the bound~(\ref{boundIu}), after choosing $\bar{T}$ large
  enough. 
\end{proof}

\subsection{Proof of absolute continuity}

In this section we prove that the measure $\mu_T$ is absolutely continuous
with respect to the measure $\mathbbm{Q}^u$ we constructed in
Lemma~\ref{Lemma:changeofvariables}. First recall that the measures $\mu_T$
defined on $\Omega$ as
\[ \frac{\mathd \mu_T}{\mathd \mathbbm{P}} = e^{- V_T (W_T)} \]
can be described, using Lemma~\ref{Lemma:changeofvariables}, as a perturbation
of $\mathbbm{Q}^u$ with density $D_T$ given by
\[ D_T \assign \left. \frac{\mathd \mu_T}{\mathd \mathbbm{Q}^u}
   \right|_{\CF_T} = \left. \frac{\mathd \mu_T}{\mathd \mathbbm{P}}
   \right|_{\CF_T} \left. \frac{\mathd \mathbbm{P}}{\mathd \mathbbm{Q}^u}
   \right|_{\CF_T} = e^{- V_T (W_T) - \int^T_0 u \mathd X + \frac{1}{2}
   \int^T_0 \| u_t \|^2_{L^2} \mathd t}, \]
at least on $\CF_T$.

\begin{lemma}
  \label{lemma:localestimate}There exists a $p > 1$, such that for any $K >
  0$,
  \[ \sup_T \mathbbm{E}_{\mathbbm{Q}^u} \left[ | D_T |^p \mathbbm{1}_{\left\{
     \| W_{\infty} \|_{\VV^{- 1 / 2 - \varepsilon}} \leqslant K \right\}}
     \right] < \infty . \]
  in particular, the family $(D_T)_T$ is uniformly integrable under
  $\mathbbm{Q}^u$.
\end{lemma}

\begin{proof}
  The proof of the first claim is given in Section~\ref{sec:lp-bounds} below.
  For the second claim fix $\varepsilon > 0$. Our aim is to show that there
  there exists $\delta > 0$ such that $\mathbbm{Q}^u (A) < \delta$ implies
  $\int_A D_T \mathd \mathbbm{Q}^u < \varepsilon$. From corollary
  \ref{cor:tightness} for any $\varepsilon > 0$ there exists a $K > 0$ such
  that
  \[ \varepsilon / 2 > \mu_T \left( \left\{ \| W_{\infty} \|_{\VV^{- 1 / 2 -
     \varepsilon}} \geqslant K \right\} \right) = \int_{\left\{ \| W_{\infty}
     \|_{\VV^{- 1 / 2 - \varepsilon}} \geqslant K \right\}} D_T \mathd
     \mathbbm{Q}^u . \]
  Then for any $A \in \CF$ such that $\mathbbm{Q}^u (A)^{(p - 1) / p} <
  \varepsilon / \left( 2 \sup_T \mathbbm{E}_{\mathbbm{Q}^u} \left[ | D_T |^p
  \mathbbm{1}_{\left\{ \| W_{\infty} \|_{\VV^{- 1 / 2 - \varepsilon}}
  \leqslant K \right\}} \right] \right)$
  \begin{eqnarray*}
    &  & \int_A D_T \mathd \mathbbm{Q}^u\\
    & = & \int_{A \cap \left\{ \| W_{\infty} \|_{\VV^{- 1 / 2 - \varepsilon}}
    \geqslant K \right\}} D_T \mathd \mathbbm{Q}^u + \int_{A \cap \left\{ \|
    W_{\infty} \|_{\VV^{- 1 / 2 - \varepsilon}} \leqslant K \right\}} D_T
    \mathd \mathbbm{Q}^u\\
    & \leqslant & \varepsilon / 2 + \sup_T \mathbbm{E}_{\mathbbm{Q}^u} \left[
    | D_T |^p \mathbbm{1}_{\left\{ \| W_{\infty} \|_{\VV^{- 1 / 2 -
    \varepsilon}} \leqslant K \right\}} \right] \mathbbm{Q}^u (A)^{(p - 1) /
    p}\\
    & \leqslant & \varepsilon
  \end{eqnarray*}
\end{proof}

\begin{corollary}
  The family of measures $(\mu_T)_{T \geqslant 0}$ is sequentially compact
  w.r.t. strong convergence on $\left( \Omega, \CF \right)$. Furthermore any
  accumulation point is absolutely continuous with respect to $\mathbbm{Q}^u$.
\end{corollary}

\begin{proof}
  We choose a subsequence (not relabeled) such that $D_T \rightarrow
  D_{\infty}$ weakly in $L^1 (\mathbbm{Q}^u)$, for some $D_{\infty} \in L^1
  (\mathbbm{Q}^u)$. It always exists by uniform integrability. We now claim
  that for any $A \in \CF$
  \[ \lim_{T \rightarrow \infty} \mu_T (A) = \int_A D_{\infty} \mathd
     \mathbbm{Q}^u_{} . \]
  It is enough to check this for $A \in \CF_S$ for any $S \in \mathbbm{R}_+$
  since these generate $\CF$. But there we have for $T \geqslant S$,
  \[ \mu_T (A) = \int_A D_T \mathd \mathbbm{Q}^u_{} \rightarrow \int_A
     D_{\infty} \mathd \mathbbm{Q}^u_{} \]
  by weak $L^1$ convergence. \ 
\end{proof}

Recall that the $\Phi_3^4$ measure can be defined as a weak limit of the
measures $\tilde{\mu}_T$ on $\VV^{- 1 / 2 - \varepsilon}$ given by
\[ \int f (\varphi) \tilde{\mu}_T (\mathd \varphi) = \int f (\varphi) e^{-
   V_T (\varphi)} \theta_T (\mathd \varphi) =\mathbbm{E}_{\mathbbm{P}} [f
   (W_T) e^{- V_T (W_T)}] \]
where $\theta_T$ is the gaussian measure with covariance $\rho^2_T (\mathD)
\langle \mathD \rangle^{- 2}$. From this together with the above
considerations we see that any accumulation point $\tilde{\mu}_{\infty}$ of
$\tilde{\mu}_T$ satisfies
\begin{equation}
  \tilde{\mu}_{\infty} (A) =\mathbbm{E}_{\mathbbm{Q}^u} [\mathbbm{1}_A
  (W_{\infty}) D_{\infty}] \label{eq:Phi4def}
\end{equation}
for some $D_{\infty} \in L^1 (\mathbbm{Q}^u)$.

\subsection{$L^p$ bounds}\label{sec:lp-bounds}

Now we will prove local $L^p$-bounds on the density $D_T$. In the sequel we
will denote $\tilde{W} = W^u$, with $u$ satisfying (\ref{eq:drift}), namely $u
= U (\tilde{W})$. Before we proceed let us study how the functional $U
(\tilde{W})$ behaves under shifts of $\tilde{W}$, since later we will want to
apply the Bou{\'e}--Dupuis formula and this kind of behavior will be crucial.
Let $w \in L^2 ([0, \infty) \times \Lambda)$ and denote
\[ u^w \assign U (\tilde{W} + I (w)) \quad \text{and} \qquad h^w \assign U
   (\tilde{W} + I (w)) + w = u^w + w. \]
The process $h^w$ satisfies
\[ h^w - w = u^w = \Xi (\tilde{W} + I (w), u^w) . \]
More explicitly, for all $s \geqslant 0$ we have
\[ \begin{array}{lll}
     h^w_s - w_s & = & - 4 \lambda J_s \llbracket \tilde{W}_s^3 \rrbracket -
     12 \lambda J_s \llbracket \tilde{W}_s^2 \rrbracket I_s (w) - 12 \lambda
     J_s \tilde{W}_s (I_s (w))^2 - 4 \lambda J_s (I_s (w))^3\\
     &  & - 12 \lambda \mathbbm{1}_{\{ s \geqslant \bar{T} \}} J_s
     (\llbracket \tilde{W}^2_s \rrbracket \succ_{} I_s^{\flat} (u^w)) - 24
     \lambda \mathbbm{1}_{\{ s \geqslant \bar{T} \}} (J_s (\tilde{W}_s I_s (w)
     \succ I^{\flat}_s (u^w)))\\
     &  & - 12 \lambda \mathbbm{1}_{\{ s \geqslant \bar{T} \}} J_s ((I_s
     (w))^2 \succ I^{\flat}_s (u^w))\\
     &  & - \sum_{i = 0}^n \binom{n}{i} J_s \llbracket (\langle \mathD
     \rangle^{- 1 / 2} \tilde{W}_s)^i \rrbracket (\langle \mathD \rangle^{- 1
     / 2} I_s (w))^{n - i} .
   \end{array} \]
Decomposing
\[ \llbracket \tilde{W}_s^2 \rrbracket I_s (w) = \llbracket \tilde{W}_s^2
   \rrbracket \succ \theta_s I_s (w) + \llbracket \tilde{W}_s^2 \rrbracket
   \succ (1 - \theta_s) I_s (w) + \llbracket \tilde{W}_s^2 \rrbracket \circ
   I_s (w) + \llbracket \tilde{W}_s^2 \rrbracket \prec I_s (w), \]
we can write
\begin{equation}
  u^w = U (\tilde{W} + I (w)) = - 4 \lambda J_s \llbracket \tilde{W}_s^3
  \rrbracket - 12 \lambda J_s (\llbracket \tilde{W}_s^2 \rrbracket \succ
  I_s^{\flat} (h^w)) + r_s^w, \label{eq:uw}
\end{equation}
with
\begin{equation}
  \begin{array}{lll}
    r_s^w & = & - 12 \lambda J_s \llbracket \tilde{W}_s^2 \rrbracket \succ (1
    - \theta_s) I_s (w) - 12 \lambda J_s (\llbracket \tilde{W}_s^2 \rrbracket
    \circ I_s (w)) - 12 \lambda J_s \llbracket \tilde{W}_s^2 \rrbracket \prec
    I_s (w)\\
    &  & - 12 \lambda J_s \tilde{W}_s (I_s (w))^2 - 4 \lambda J_s (I_s (w))^3
    - 24 \lambda \mathbbm{1}_{\{ s \geqslant \bar{T} \}} (J_s (\tilde{W}_s I_s
    (w) \succ \theta_s I^{\flat}_s (u^w)))\\
    &  & - 12 \lambda \mathbbm{1}_{\{ s \geqslant \bar{T} \}} J_s ((I_s
    (w))^2 \succ I^{\flat}_s (u^w)) + 12 \lambda \mathbbm{1}_{\{ s < \bar{T}
    \}} J_s (\llbracket \tilde{W}_s^2 \rrbracket \succ I_s^{\flat} (u^w))\\
    &  & - \sum_{i = 0}^n \binom{n}{i} J_s \langle \mathD \rangle^{- 1 / 2}
    [\llbracket (\langle \mathD \rangle^{- 1 / 2} W_s)^i \rrbracket (\langle
    \mathD \rangle^{- 1 / 2} I_s (w))^{n - i}_{}] .
  \end{array} \label{eq:rw}
\end{equation}
The first two terms in~(\ref{eq:uw}) will be used for renormalization while
the remainder $r^w$ contains terms of higher regularity which will have to be
estimated in the sequel.

\begin{proof*}{Proof of Lemma~\ref{lemma:localestimate}}
  Observe that
  \[ \mathbbm{1}_{\left\{ \| W_{\infty} \|_{\VV^{- 1 / 2 - \varepsilon}}
     \leqslant K \right\}} \lesssim_{K, n} \exp \left( - \| W_{\infty}
     \|^n_{\VV^{- 1 / 2 - \varepsilon}} \right) = \exp \left( - \|
     \tilde{W}_{\infty} + I_{\infty} (U (\tilde{W})) \|^n_{\VV^{- 1 / 2 -
     \varepsilon}} \right) \]
  and
  \[ | D_T |^p = e^{- p \left[ V_T (\tilde{W}_T + I (U (\tilde{W}))) +
     \int^T_0 U (\tilde{W}) \mathd \tilde{X} + \frac{1}{2} \int^T_0 \| U_t
     (\tilde{W}) \|^2_{L^2} \mathd t \right]} . \]
  Combining these two facts we have
  \[ \begin{array}{l}
       \\
       \mathbbm{E}_{\mathbbm{Q}^u} \left[ | D_T |^p \mathbbm{1}_{\left\{ \| W
       \|_{\VV^{- 1 / 2 - \varepsilon}} \leqslant K \right\}} \right]\\
       \lesssim \mathbbm{E}_{\mathbbm{Q}^u} \left[ \exp \left( - p \left( V_T
       (\tilde{W}_T + I_T (U (\tilde{W}))) + \int^T_0 U_t (\tilde{W}) \mathd
       \widetilde{X_t} + \frac{1}{2} \int^T_0 \| U_t (\tilde{W}) \|^2_{L^2}
       \mathd t \right) \right. \right.\\
       \qquad \qquad \qquad \qquad \qquad \left. \left. - \|
       \tilde{W}_{\infty} + I_{\infty} (U (\tilde{W})) \|^n_{\VV^{- 1 / 2 -
       \varepsilon}} \right) \right]\\
       =\mathbbm{E} \left[ \exp \left( - p \left( V_T (W_T + I_T (U (W))) +
       \int^T_0 U_t (W) \mathd X_t + \frac{1}{2} \int^T_0 \| U_t (W)
       \|^2_{L^2} \mathd t \right) \right. \right.\\
       \qquad \qquad \qquad \qquad \qquad \left. \left. - \| W_{\infty} +
       I_{\infty} (U (W)) \|^n_{\VV^{- 1 / 2 - \varepsilon}} \right) \right] .
     \end{array} \]
  The Bou{\'e}--Dupuis formula~(\ref{eq:bd}) provides the variational bound
  \[ \begin{array}{l}
       - \log \mathbbm{E}_{\mathbbm{Q}^u} \left[ | D_T |^p
       \mathbbm{1}_{\left\{ \| W \|_{\VV^{- 1 / 2 - \varepsilon}} \leqslant K
       \right\}} \right]\\
       \qquad \gtrsim \inf_{w \in \mathbbm{H}_a} \mathbbm{E} \left[ p \left(
       V_T (W_T + I_T (h^w)) + \frac{1}{2} \int^T_0 \| h^w \|^2_{L^2} \mathd t
       \right) \right.\\
       \qquad \qquad \hspace{3em} \left. + \frac{1 - p}{2} \int^T_0 \| w_t
       \|^2_{L^2} \mathd t + \| W_{\infty} + I_{\infty} (h^w) \|^n_{\VV^{- 1 /
       2 - \varepsilon}} + \frac{1}{2} \int^{\infty}_T \| w_t \|^2_{L^2}
       \mathd t \right]
     \end{array} \]
  where we have set $h^w = w + U (W + I (w))$ as above. Recall now that from
  Theorem~\ref{thm:equicoerc} there exists a constant $C$, independent of $T$
  , such that for each $h^w$,
  \[ \mathbbm{E} \left[ p \left( V_T (W_T + I_T (h^w)) + \frac{1}{2} \int^T_0
     \| h^w \|^2_{L^2} \mathd t \right) \right] \geqslant - C + \frac{1}{4}
     \mathbbm{E}_{\mathbbm{P}} \left[ \lambda \| I_T (h^w) \|^4_{L^4} +
     \int^T_0 \| l^T (h^w) \|_{L^2}^2 \right] \]
  where
  \[ l_t^T (h^w) = h_t^w + \lambda \mathbbm{1}_{t \leqslant T}
     \mathbbm{W}^{\langle 3 \rangle}_t + \lambda \mathbbm{1}_{t \leqslant T}
     J_t (\mathbbm{W}_t^2 \succ I_t^{\flat} (h^w)) . \]
  Using eq.~(\ref{eq:uw}) we compute
  \begin{eqnarray*}
    \mathbbm{1}_{t \leqslant T} l_t^T (h^w) & = & \mathbbm{1}_{t \leqslant T}
    h_t^w + \lambda \mathbbm{1}_{t \leqslant T} \mathbbm{W}^{\langle 3
    \rangle}_t + \lambda \mathbbm{1}_{t \leqslant T} J_t (\mathbbm{W}_t^2
    \succ I_t^{\flat} (h^w))\\
    & = & \mathbbm{1}_{t \leqslant T} (u_t^w + w_t) + \lambda \mathbbm{1}_{t
    \leqslant T} \mathbbm{W}^{\langle 3 \rangle}_t + \lambda \mathbbm{1}_{t
    \leqslant T} J_t (\mathbbm{W}_t^2 \succ I_t^{\flat} (h^w))\\
    & = & \mathbbm{1}_{t \leqslant T} (r_t^w + w_t) .
  \end{eqnarray*}
  At this point we need a lower bound for
  \[ \begin{array}{l}
       \mathbbm{E} \left[ \frac{1}{4} \left( \lambda \| I_T (h^w) \|^4_{L^4} +
       \int^T_0 \| r_t^w + w_t \|_{L^2}^2 \mathd t \right) + \frac{1 - p}{2}
       \int^T_0 \| w_t \|^2_{L^2} \mathd t \right.\\
       \qquad \qquad \qquad \qquad \qquad \left. + \| W_{\infty} + I_{\infty}
       (h^w) \|^n_{\VV^{- 1 / 2 - \varepsilon}} + \frac{1}{2} \int^{\infty}_T
       \| w_t \|^2_{L^2} \mathd t \right] - C.
     \end{array} \]
  Given that we need to take $p > 1$, this expression present a difficulty in
  the fact that the term $\int^T_0 \| w_t \|^2_{L^2} \mathd t$ appears with a
  negative coefficient. Note that this term cannot easily be controlled via
  $\int^T_0 \| r_t^w + w_t \|_{L^2}^2 \mathd t$ since the contribution $r^w$,
  see eq.~(\ref{eq:rw}), contains factors which are homogeneous in $w$ of
  order up to $3$. This is the reason we had to localize the estimate,
  introduce the ``good'' term $\| W_{\infty} + I_{\infty} (h^w) \|^n_{\VV^{- 1
  / 2 - \varepsilon}}$, and introduce the term $J_s \langle \mathD \rangle^{-
  1 / 2} (\llbracket (\langle \mathD \rangle^{- 1 / 2} W_s)^n \rrbracket)$ in
  (\ref{eq:pre-pre-drift}) which will help us to control the growth of $r^w$.
  Indeed in Lemma~\ref{lemma:driftgronwallestimate} below, a Gronwall argument
  will allow us to show that $\int^T_0 \| w_t \|^2_{L^2} \mathd t$ can be
  bounded by a combination of the other ``good'' terms as
  \[ \mathbbm{E} \left[ \int^T_0 \| w \|^2_{L^2} \mathd t \right] \lesssim
     \mathbbm{E} \left[ \| I^{\flat}_T (h) \|^4_{L^4} + \| I_T^{\flat} (h)
     \|^n_{\VV^{- 1 / 2 - \varepsilon}} + \int^T_0 \| w_t + r_t^w \|_{L^2}^2
     \mathd t + 1 \right] . \]
  This implies that for $1 < p \ll 2$,
  \[ \begin{array}{l}
       - \log \mathbbm{E}_{\mathbbm{Q}^u} \left[ | D_T |^p
       \mathbbm{1}_{\left\{ \| W \|_{\VV^{- 1 / 2 - \varepsilon}} \leqslant K
       \right\}} \right]\\
       \quad \geqslant \inf_{w \in \mathbbm{H}_a} \mathbbm{E} \left\{
       \frac{1}{4} \left[ \lambda \| I_T (h^w) \|^4_{L^4} + \int^T_0 \| l^T_t
       (h^w) \|_{L^2}^2 \mathd t \right] \right.\\
       \quad \quad \quad + (1 - p) C \left[ \| I^{\flat}_T (h^w) \|^4_{L^4} +
       \| I_T^{\flat} (h^w) \|^n_{\VV^{- 1 / 2 - \varepsilon}} + \int^T_0 \|
       l^T_t (h^w) \|_{L^2}^2 \mathd t \right]\\
       \quad \quad \quad \left. + \| W_{\infty} + I_{\infty} (h^w)
       \|^n_{\VV^{- 1 / 2 - \varepsilon}} \right\} - C\\
       \quad \geqslant - C
     \end{array} \]
  which gives the claim. Note that here we used the bound
  \[ \mathbbm{E} \| I_{\infty} (h^w) \|^n_{\VV^{- 1 / 2 - \varepsilon}}
     \lesssim \mathbbm{E} \| W_{\infty} \|^n_{\VV^{- 1 / 2 - \varepsilon}}
     +\mathbbm{E} \| W_{\infty} + I_{\infty} (h^w) \|^n_{\VV^{- 1 / 2 -
     \varepsilon}} \]
  \[ \lesssim C +\mathbbm{E} \| W_{\infty} + I_{\infty} (h^w) \|^n_{\VV^{- 1 /
     2 - \varepsilon}} \]
  as well as the fact that $\| I^{\flat}_t (h^w) \|_{\VV^{- 1 / 2 -
  \varepsilon}} \lesssim \| I_{\infty} (h^w) \|_{\VV^{- 1 / 2 - \varepsilon}}$
  to conclude.
\end{proof*}

The following lemmas complete the proof.

\begin{lemma}
  \label{lemma:driftgronwallestimate}For $n \in \mathbbm{N}$ odd and large
  enough
  \[ \mathbbm{E} \int^T_0 \| w_s \|_{L^2}^2 \mathd s \lesssim \mathbbm{E}
     \int^T_0 \| w_s + r_s^w \|^2 \mathd s +\mathbbm{E} \| I_T^{\flat} (h^w)
     \|^{n + 1}_{\VV^{- 1 / 2 - \varepsilon}} + \| I^{\flat}_T (h^w)
     \|^4_{L^4} + 1. \]
\end{lemma}

\begin{proof}
  Let us introduce the notation
  \[ \tmop{Aux}_s (W, w) \assign \sum_{i = 0}^n \binom{n}{i} J_s \langle
     \mathD \rangle^{- 1 / 2} (\llbracket (\langle \mathD \rangle^{- 1 / 2}
     W_s)^i \rrbracket (\langle \mathD \rangle^{- 1 / 2} I_s (w))^{n - i}_{})
     . \]
  Write $r_s^w = \tilde{r}_s^w + \tmop{Aux}_s (W, w)$ and observe that
  \begin{eqnarray*}
    w^2_s & = & 2 (w_s + r_s^w)^2 - 4 w_s r_s^w - 2 (r_s^w)^2 - w^2_s\\
    & = & 2 (w_s + r_s^w)^2 - 4 w_s \tilde{r}_s^w - 2 (r_s^w)^2 - w^2_s - 4
    \tmop{Aux}_s (W, w) w_s .
  \end{eqnarray*}
  We can apply Ito formula to obtain
  \[ \int^T_0 \int_{\Lambda} \tmop{Aux}_s (W, w) w_s \tmop{ds} =
     \overline{\tmop{Aux}}_T (W, w) + \text{martingale} \]
  where
  \[ \overline{\tmop{Aux}}_T (W, w) \assign \sum_{i = 0}^n \frac{1}{n + 1 - i}
     \binom{n}{i} \int_{\Lambda} (\llbracket (\langle \mathD \rangle^{- 1 / 2}
     W_T)^i \rrbracket (\langle \mathD \rangle^{- 1 / 2} I_T (w))^{n + 1 -
     i}_{}) . \]
  By Lemma~\ref{lemma:bounds2} below, we have constants $c, C$ and a random
  variable $Q_T (W)$ such that
  \[ \sup_T \mathbbm{E} [| Q_T (W) |] < \infty \]
  and

  \[ c \int^T_0 \| w_s \|_{L^2}^2 \mathd s + c \| I_T (w) \|^{n + 1}_{W^{- 1 /
     2, n + 1}} - Q_T (W) \leqslant \int^T_0 \| w_s \|_{L^2}^2 \mathd s +
     \overline{\tmop{Aux}}_T (W, w) \]
  \[ \leqslant C \| I_T (w) \|^{n + 1}_{W^{- 1 / 2, n + 1}} + C \int^T_0 \|
     w_s \|_{L^2}^2 \mathd s + Q_T (W) . \]
  Integrating over space we obtain
  \begin{eqnarray*}
    \frac{\mathd}{\mathd t} \mathbbm{E} \left( \int^t_0 \| w_s \|^2 \mathd s +
    \overline{\tmop{Aux}}_t (W, w) \right) & \leqslant & 2\mathbbm{E} \| w_t +
    r_t^w \|_{L^2}^2 - 4\mathbbm{E} \int w_s \tilde{r}_s^w - \| w_s
    \|_{L^2}^2\\
    & \leqslant & 2\mathbbm{E} \| w_t + r_t^w \|_{L^2}^2 + 4\mathbbm{E} \|
    \tilde{r}_s^w \|_{L^2}^2 .
  \end{eqnarray*}
  Now by Lemma~\ref{lemma:estimateremainder} below
  \[ \begin{array}{lll}
       \langle t \rangle^{1 + \varepsilon} \| \tilde{r}_t^w \|_{L^2}^2 &
       \lesssim & \int^t_0 \| w_s \|_{L^2}^2 \mathd s + \| I_t (w) \|^{n +
       1}_{W^{- 1 / 2, n + 1}} + \| I_t^{\flat} (h^w) \|^{n + 1}_{\VV^{- 1 / 2
       - \varepsilon}} + \| I^{\flat}_t (h^w) \|^4_{L^4} +_{} Q_t (W)\\
       & \lesssim & Q_t (W) + \int^t_0 \| w_s \|_{L^2}^2 \mathd s +
       \overline{\tmop{Aux}}_t (W, w) + \| I_t^{\flat} (h^w) \|^{n +
       1}_{\VV^{- 1 / 2 - \varepsilon}} + \| I^{\flat}_t (h^w) \|^4_{L^4}
     \end{array} \]
  Gathering things together we have
  \[ \begin{array}{ll}
       & \frac{\mathd}{\mathd t} \mathbbm{E} \left( \int^t_0 \| w_s \|^2
       \mathd s + \overline{\tmop{Aux}}_t (W, w) \right)\\
       \lesssim & \frac{1}{\langle t \rangle^{1 + \varepsilon}} \left(
       \int^t_0 \| w_s \|_{L^2}^2 \mathd s + \overline{\tmop{Aux}}_t (W, w)
       \right) + \frac{1}{\langle t \rangle^{1 + \varepsilon}} \left( \|
       I_t^{\flat} (h^w) \|^{n + 1}_{\VV^{- 1 / 2 - \varepsilon}} + \|
       I^{\flat}_t (h^w) \|^4_{L^4} \right)\\
       & + 2\mathbbm{E} \| w_t + r_t^w \|_{L^2}^2 .
     \end{array} \]
  Then Gronwall's lemma allows to conclude
  \[ - 1 +\mathbbm{E} \left( \int^t_0 \| w_s \|_{L^2}^2 \mathd s \right)
     \lesssim \mathbbm{E} \left( \int^t_0 \| w_s \|_{L^2}^2 \mathd s +
     \overline{\tmop{Aux}}_t (W, w) \right) \]
  \[ \lesssim \mathbbm{E} \left( \int^t_0 \| w_s + r_s^w \|_{L^2}^2 \mathd s +
     \| I_t^{\flat} (h^w) \|^{n + 1}_{\VV^{- 1 / 2 - \varepsilon}} + \|
     I^{\flat}_t (h^w) \|^4_{L^4} \right) + 1. \]
  
\end{proof}

\begin{lemma}
  \label{lemma:bounds2}There exists constants $c, C$ and a random variable
  $Q_T (W)$ such that
  \[ \sup_T \mathbbm{E} [| Q_T (W) |] < \infty, \]
  and
  \begin{eqnarray*}
    &  & - Q_T (W) + c \int^T_0 \| w_s \|_{L^2}^2 \mathd s + c \| I_T (w)
    \|^{n + 1}_{W^{- 1 / 2, n + 1}}\\
    & \leqslant & \int^T_0 \| w_s \|_{L^2}^2 \mathd s +
    \overline{\tmop{Aux}}_T (W, w)\\
    & \leqslant & C \| I_T (w) \|^{n + 1}_{W^{- 1 / 2, n + 1}} + C \int^T_0
    \| w_s \|_{L^2}^2 \mathd s + Q_T (W)
  \end{eqnarray*}
  
\end{lemma}

\begin{proof}
  We recall that
  \begin{eqnarray*}
    \overline{\tmop{Aux}}_T (W, w) & = & \sum_{i = 0}^n \frac{1}{n + 1 - i}
    \binom{n}{i} \int (\llbracket (\langle \mathD \rangle^{- 1 / 2} W_T)^i
    \rrbracket^{} (\langle \mathD \rangle^{- 1 / 2} I_T (w))^{n + 1 - i}_{})\\
    & = & \sum_{i = 1}^n \frac{1}{n + 1 - i} \binom{n}{i} \int (\llbracket
    (\langle \mathD \rangle^{- 1 / 2} W_T)^i \rrbracket^{} (\langle \mathD
    \rangle^{- 1 / 2} I_T (w))^{n + 1 - i}_{})\\
    &  & + \frac{1}{n + 1} \| I_T (w) \|^{n + 1}_{W^{- 1 / 2, n + 1}}
  \end{eqnarray*}
  and since $\sup_{T < \infty} \mathbbm{E} \left[ \| \llbracket (\langle
  \mathD \rangle^{- 1 / 2} W_T)^i \rrbracket \|^p_{\VV^{- \varepsilon}}
  \right] < \infty$ for any $p < \infty$ and any $\varepsilon > 0$ it is
  enough to bound $\| (\langle \mathD \rangle^{- 1 / 2} I_T (w))^{n + 1 -
  i}_{} \|^q_{B_{1, 1}^{\varepsilon}}$ for some $q > 1$ by the terms $\| I_T
  (w) \|^{n + 1}_{W^{- 1 / 2, n + 1}}$ and $\| I_T (w) \|^2_{H^1} \lesssim
  \int^T_0 \| w_s \|_{L^2}^2 \mathd s$. By interpolation we can estimate, for
  $i \geqslant 1$,
  \[ \begin{array}{lll}
       \| (\langle \mathD \rangle^{- 1 / 2} I_T (w))^{n + 1 - i}_{} \|_{B_{1,
       1}^{\varepsilon}} & \lesssim & \| \langle \mathD \rangle^{- 1 / 2} I_T
       (w) \|^n_{B^{\varepsilon}_{n, 1}} + C\\
       \qquad & \lesssim & \| I_T (w) \|_{W^{- 1 / 2, n + 1}}^{n - \frac{1}{(n
       - 1)}} \| I_T (w) \|^{\frac{1}{n - 1}}_{H^1} + C \qquad \qquad
       \text{(let $\varepsilon = \tfrac{1}{n (n - 1)}$)}
     \end{array} \]
  Choosing $q = n / \left( n - \frac{1}{(n - 1)} \right) > 1$, we have
  \[ \left( \| I_T (w) \|_{W^{- 1 / 2, n + 1}}^{n - \frac{1}{(n - 1)}} \|
     I_T (w) \|^{\frac{1}{n - 1}}_{H^1} \right)^q = \| I_T (w) \|_{W^{- 1 / 2,
     n + 1}}^n \| I_T (w) \|^{\frac{n}{(n - 1) n - 1}}_{H^1} . \]
  Now for $n$ large enough $\frac{n}{(n - 1) n - 1} \leqslant \frac{2}{n + 1}$
  and using Young's inequality we can estimate
  \[ \begin{array}{lll}
       \| I_T (w) \|_{W^{- 1 / 2, n + 1}}^n \| I_T (w) \|^{\frac{n}{(n - 1)
       n - 1}}_{H^1} & \lesssim & \| I_T (w) \|_{W^{- 1 / 2, n + 1}}^n
       \left( \| I_T (w) \|^{\frac{2}{n + 1}}_{H^1} + 1 \right)\\
       & \lesssim & \| I_T (w) \|_{W^{- 1 / 2, n + 1}}^{n + 1} + \| I_T
       (w) \|^2_{H^1} + 1
     \end{array} \]
  
\end{proof}

\begin{lemma}
  \label{lemma:estimateremainder}Let
  \[ \begin{array}{lll}
       \tilde{r}_s^w & = & - 12 \lambda J_s \llbracket W_s^2 \rrbracket \succ
       (1 - \theta_s) I_s (w) + 12 \lambda J_s (\llbracket W_s^2 \rrbracket
       \circ I_s (w)) + 12 \lambda J_s \llbracket W_s^2 \rrbracket \prec I_s
       (w)\\
       &  & - 12 \lambda J_s W_s (I_s (w))^2 - 4 \lambda J_s (I_s (w))^3 - 24
       \lambda (J_s (W_s I_s (w) \succ \theta_s I^{\flat}_s (u^w)))\\
       &  & - 12 \lambda J_s ((I_s (w))^2 \succ \theta_s I^{\flat}_s (u^w)) +
       \lambda \mathbbm{1}_{\{ s < \bar{T} \}} J_s (\mathbbm{W}_s^2 \succ
       I_s^{\flat} (u^w)) .
     \end{array} \]
  Setting $h^w = u + w$, there exists a random variable $Q_t (W)$ such that
  $\sup_t \mathbbm{E} [| Q_t (W) |] < \infty$ and
  \[ \langle t \rangle^{1 + \varepsilon} \| \tilde{r}_t^w \|^2 \lesssim
     \int^t_0 \| w_s \|_{L^2}^2 \mathd s + \| I_t (w) \|^{n + 1}_{W^{- 1 / 2,
     n + 1}} + \|I_t^{\flat} (h^w) \|^{n + 1}_{\VV^{- 1 / 2 - \varepsilon}} +
     \| I^{\flat}_t (h^w) \|^4_{L^4} +_{} Q_t (W) . \]
\end{lemma}

\begin{proof}
  Note that
  \[ \begin{array}{lll}
       \| \mathbbm{1}_{\{ s < \bar{T} \}} J_s (\mathbbm{W}_s^2 \succ
       I_s^{\flat} (u^w)) \|_{L^2}^2 & \lesssim_{\bar{T}} & \frac{1}{\langle s
       \rangle^2} \| \mathbbm{W}_s^2 \|^2_{\VV^{- 1 - \varepsilon}} \|
       I_s^{\flat} (u^w) \|^2_{L^4}\\
       & \lesssim & \frac{1}{\langle s \rangle^2} \left( \| \mathbbm{W}_s^2
       \|^4_{\VV^{- 1 - \varepsilon}} + \| I_s^{\flat} (u^w) \|^4_{L^4}
       \right) .
     \end{array} \]
  Moreover $h^w = u^w + w$ implies
  \[ \|I^{\flat}_t (u^w) \|^{n + 1}_{\VV^{- 1 / 2 - \varepsilon}} \lesssim
     \|I^{\flat}_t (w) \|^{n + 1}_{\VV^{- 1 / 2 - \varepsilon}} +
     \|I_t^{\flat} (h^w) \|^{n + 1}_{\VV^{- 1 / 2 - \varepsilon}}, \]
  and $\| I^{\flat}_t (u^w) \|^4_{L^4} \lesssim_{} \| I^{\flat}_t (h^w)
  \|^4_{L^4} + \| I^{\flat}_t (w) \|^4_{L^4}$ . From Lemma~\ref{L4estimate} we
  get
  \[ \| I^{\flat}_t (w) \|^4_{L^4} \lesssim C + \int^t_0 \| w_s \|_{L^2}^2
     \mathd s + \| I_t (w) \|^{n + 1}_{W^{- 1 / 2, n + 1}} . \]
  The estimation for the other terms is easy but technical and postponed until
  Section~\ref{sec:analytic}.
\end{proof}

\section{Singularity of $\Phi_3^4$ w.r.t. the free field}

The goal of this section is to prove that the $\Phi_3^4$ measure is singular
with respect to the Gaussian free field. For this we have to find a set $S
\subseteq \VV^{- 1 / 2 - \varepsilon} (\Lambda)$ such that $\mathbbm{P}
(W_{\infty} \in S) = 1$ and $\mathbbm{Q}^u (W_{\infty} \in S) = 0$. Together
with~(\ref{eq:Phi4def}), this will imply singularity. We claim that setting
\[ S \assign \left\{ f \in \VV^{- 1 / 2 - \varepsilon} (\Lambda) :
   \frac{1}{T_n^{1 / 2 + \delta}} \int_{\Lambda} \llbracket (\theta_{T_n} f)^4
   \rrbracket \rightarrow 0 \right\} \]
for some suitable subsequence $T_n$, does the job. Here
\[ \llbracket (\theta_T f)^4 \rrbracket = (\theta_T f)^4 - 6\mathbbm{E}
   [(\theta_T W_{\infty} (0))^2] (\theta_T f)^2 + 3\mathbbm{E} [(\theta_T
   W_{\infty} (0))^2]^2 \]
denotes the Wick ordering with respect to the Gaussian free field. Let us
prove first that indeed $\mathbbm{P} (W_{\infty} \in S) = 1$ for some $T_n$.
For later use we define
\[ \mathbbm{W}_t^{\theta_T, 3} = 4 (\theta_T W_t)^3 - 12\mathbbm{E} [(\theta_T
   W_t (0))^2] (\theta_T W_t) \]
and
\[ \mathbbm{W}_t^{\theta_T, 2} = 12 ((\theta_T W_t)^2 -\mathbbm{E} [(\theta_T
   W_t (0))^2]) . \]
\begin{lemma}
  For any $\delta > 0$
  \[ \lim_{T \rightarrow \infty} \mathbbm{E} \left[ \left( \frac{1}{T^{(1 +
     \delta) / 2}} \int_{\Lambda} \llbracket (\theta_T W_{\infty})^4
     \rrbracket \right)^2 \right] = 0. \]
\end{lemma}

\begin{proof}
  Wick products corresponds to iterated Ito integrals. Introducing the
  notation
  \[ \mathd w_t^{\theta_T} = \theta_T J_t \mathd X_t, \]
  we can verify by Ito formula that
  \[ \int_{\Lambda} \llbracket \theta_T W_{\infty}^4 \rrbracket =
     \int^{\infty}_0 \int_{\Lambda} \mathbbm{W}_t^{\theta_T, 3} \mathd
     w^{\theta_T}_t = \int^{\infty}_0 \int_{\Lambda} \theta_T J_t
     \mathbbm{W}_t^{\theta_T, 3} \mathd X_t . \]
  Since $\theta_T J_t = 0$ for $t \geqslant T$, Ito isometry gives
  \[ \mathbbm{E} \left| \int^{\infty}_0 \int_{\Lambda} \theta_T J_t
     \mathbbm{W}_t^{\theta_T, 3} \mathd X_t \right|^2 =\mathbbm{E} \int^T_0
     \int_{\Lambda} (\theta_T J_t \mathbbm{W}_t^{\theta_T, 3})^2 \mathd t. \]
  Then, again by Ito formula the expectation on the r.h.s. can be estimated as
  \[\begin{array}{lll}
    \mathbbm{E} \left[ \int_{\Lambda} (\mathbbm{W}_t^{\theta_T, 3})^2 \right]
    & = & 4\mathbbm{E} \left[ \left| \sum_{k_1, k_2, k_3} \int^t_0
    \int^{s_1}_0 \int^{s_2}_0 \mathd w_{s_1}^{\theta_T} (k_1) \mathd
    w_{s_2}^{\theta_T} (k_2) \mathd w^{\theta_T}_{s_3} (k_3) \right|^2
    \right]\\
    & = & 24\mathbbm{E} \left[ \sum_{k_1, k_2, k_3} \int^t_0 \int^{s_1}_0
    \int^{s_2}_0 \frac{\theta^2_T (k_1) \sigma^2_{s_1} (k_1)_{}}{\langle k_1
    \rangle^2} \frac{\theta_T^2 (k_2) \sigma^2_{s_2} (k_2)}{\langle k_2
    \rangle^2} \frac{\theta^2_T (k_3) \sigma^2_{s_3} (k_3)}{\langle k_3
    \rangle^2} \mathd s_1 \mathd s_2 \mathd s_3 \right]\\
    & \leqslant & 24\mathbbm{E} \left[ \sum_{k_1, k_2, k_3} \int^t_0 \int^t_0
    \int^t_0 \frac{\sigma^2_{s_1} (k_1)_{}}{\langle k_1 \rangle^2}
    \frac{\sigma^2_{s_2} (k_2)}{\langle k_2 \rangle^2} \frac{\sigma^2_{s_3}
    (k_3)}{\langle k_3 \rangle^2} \mathd s_1 \mathd s_2 \mathd s_3 \right]\\
    & \lesssim & t^3
  \end{array}
  \]
  Now recall that $\| J_t f \|_{L^2 (\Lambda)} \lesssim \langle t \rangle^{- 3
  / 2} \| f \|_{L^2 (\Lambda)}$ to conclude:
  \[ \mathbbm{E} \left[ \frac{1}{T^{1 + \delta}} \int^T_0 \int_{\Lambda}
     (\theta_T J_t \mathbbm{W}_t^{\theta_T, 3})^2 \mathd t \right] \leqslant
     \frac{1}{T^{1 + \delta}} \int^T_0 \frac{1}{t^3} \mathbbm{E} [\| (\theta_T
     \mathbbm{W}_t^{\theta_T, 3}) \|_{L^2 (\Lambda)}^2] \mathd t \rightarrow 0
  \]
\end{proof}

The lemma implies that $\frac{1}{T^{(1 + \delta) / 2}} \int_{\Lambda}
\llbracket (\theta_T W_{\infty})^4 \rrbracket \rightarrow 0$ in $L^2
(\mathbbm{P})$. So there exists a subsequence $T_n$ such that
$\frac{1}{T_n^{(1 + \delta) / 2}} \int_{\Lambda} \llbracket (\theta_{T_n}
W_{\infty})^4 \rrbracket \rightarrow 0$ almost surely.

\

The next step of the proof is to check that $\mathbbm{Q}^u (W_{\infty} \in S)
= 0$. More concretely we will show that for a subsequence of $T_n$ (not
relabeled)
\[ \frac{1}{T_n^{1 - \delta}} \int_{\Lambda} \llbracket (\theta_{T_n}
   W_{\infty})^4 \rrbracket \rightarrow - \infty, \]
$\mathbbm{Q}^u$ almost surely. Observe that
\[ \begin{array}{lll}
     \int_{\Lambda} \llbracket (\theta_T W_{\infty})^4 \rrbracket & = &
     \int^{\infty}_0 \int_{\Lambda} \theta_T J_t \mathbbm{W}_t^{\theta_T, 3}
     \mathd X_t\\
     & = & \int^{\infty}_0 \int_{\Lambda} \theta_T J_t
     \mathbbm{W}_t^{\theta_T, 3} \mathd X^u_t + \int^{\infty}_0 \int_{\Lambda}
     \theta_T J_t \mathbbm{W}_t^{\theta_T, 3} u_t \mathd t\\
     & = & \int^{\infty}_0 \int_{\Lambda} \theta_T J_t
     \mathbbm{W}_t^{\theta_T, 3} \mathd X^u_t - \lambda \int^{\infty}_0
     \int_{\Lambda} (\theta_T J_t \mathbbm{W}_t^{\theta_T, 3}) J_t
     \mathbbm{W}_t^{u, 3} \mathd t\\
     &  & - \lambda \int^{\infty}_{\bar{T}} \int_{\Lambda} (\theta_T J_t
     \mathbbm{W}_t^{\theta_T, 3}) J_t (\mathbbm{W}_t^{u, 2} \succ I^{\flat}_t
     (u)) \mathd t\\
     &  & - \int^{\infty}_0 \int_{\Lambda} (\theta_T J_t
     \mathbbm{W}_t^{\theta_T, 3}) J_t \langle D \rangle^{- 1 / 2} \llbracket
     (\langle D \rangle^{- 1 / 2} W^u_t)^n \rrbracket \mathd t.
   \end{array} \]
\[ \  \]
We expect the term
\[ \int^{\infty}_0 \int_{\Lambda} (\theta_T J_t \mathbbm{W}_t^{\theta_T, 3})
   J_t \mathbbm{W}_t^{u, 3} \mathd t \]
to go to infinity faster than $T^{1 - \delta}$, $\mathbbm{Q}^u$-almost surely.
To actually prove it, we start by a computation in average.

\begin{lemma}
  It holds
  \[ \lim_{T \rightarrow \infty} \frac{1}{T^{1 - \delta}} \mathbbm{E} \left[
     \int^{\infty}_0 \int_{\Lambda} (\theta_T J_t \mathbbm{W}_t^{\theta_T, 3})
     J_t \mathbbm{W}_t^3 \mathd t \right] = \infty . \]
\end{lemma}

\begin{proof}
  Recall that $\mathd w_t^{\theta_T} = \theta_T J_t \mathd X_t$. With a slight
  abuse of notation we can write
  
  \[ \begin{array}{lll}
       &  & \int^{\infty}_0 \int_{\Lambda} (\theta_T J_t
       \mathbbm{W}_t^{\theta_T, 3}) J_t \mathbbm{W}_t^3 \mathd t\\
       & = & 16 \int^{\infty}_0 \sum_k \frac{\theta_T (k) \sigma^2_t
       (k)}{\langle k \rangle^2} \left( \sum_{k_1 + k_2 + k_3 = k} \int^t_0
       \int^{s_1}_0 \int^{s_2}_0 \mathd w_{s_1}^{\theta_T} (k_1) \mathd
       w_{s_2}^{\theta_T} (k_2) \mathd w^{\theta_T}_{s_3} (k_3) \right.\\
       &  & \times \left. \sum_{k_1 + k_2 + k_3 = k} \int^t_0 \int^{s_1}_0
       \int^{s_2}_0 \mathd w_{s_1} (k_1) \mathd w_{s_2} (k_2) \mathd w_{s_3}
       (k_3)  \right) \mathd t
     \end{array}  \]
  
  and by Ito isometry
  \[ \begin{array}{lll}
       &  & \mathbbm{E} \left. \Biggl[ \sum_{k_1 + k_2 + k_3 = k} \int^t_0
       \int^{s_1}_0 \int^{s_2}_0 \mathd w_{s_1}^{\theta_T} (k_1) \mathd
       w_{s_2}^{\theta_T} (k_2) \mathd w^{\theta_T}_{s_3} (k_3)  \right.\\
       &  & \times \sum_{k_1 + k_2 + k_3 = k} \left. \int^t_0 \int^{s_1}_0
       \int^{s_2}_0 \mathd w_{s_1} (k_1) \mathd w_{s_2} (k_2) \mathd w_{s_3}
       (k_3) \right]\\
       & = & 6 \sum_{k_1 + k_2 + k_3 = k} \int^t_0 \int^{s_1}_0 \int^{s_2}_0
       \frac{\theta_T (k_1) \sigma^2_{s_1} (k_1)}{\langle k_1 \rangle^2}
       \frac{\theta_T (k_2) \sigma^2_{s_2} (k_2)}{\langle k_2 \rangle^2}
       \frac{\theta_T (k_3) \sigma^2_{s_3} (k_3)}{\langle k_3 \rangle^2}
       \mathd s_1 \mathd s_2 \mathd s_3
     \end{array} \]
  For $T$ large enough and since $\sigma^2$ and $\theta$ are positive, we have
 {\footnotesize
 \[ \begin{array}{lll}
    &  & \int^{\infty}_0 \sum_k \frac{\theta_T (k) \sigma^2_t (k)}{\langle k
    \rangle^2} \sum_{k_1 + k_2 + k_3 = k} \int^t_0 \int^{s_1}_0 \int^{s_2}_0
    \frac{\theta_T (k_1) \sigma^2_{s_1} (k_1)}{\langle k_1 \rangle^2}
    \frac{\theta_T (k_2) \sigma^2_{s_2} (k_2)}{\langle k_2 \rangle^2}
    \frac{\theta_T (k_3) \sigma^2_{s_3} (k_3)}{\langle k_3 \rangle^2} \mathd
    s_1 \mathd s_2 \mathd s_3 \mathd t\\
    & \geqslant & \int^{T / 2}_{T / 8} \sum_k \frac{\sigma^2_t (k)}{\langle k
    \rangle^2} \sum_{k_1 + k_2 + k_3 = k} \int^{T / 8}_0 \int^{s_1}_0
    \int^{s_2}_0 \frac{\sigma^2_{s_1} (k_1)}{\langle k_1 \rangle^2}
    \frac{\sigma^2_{s_2} (k_2)}{\langle k_2 \rangle^2} \frac{\sigma^2_{s_3}
    (k_3)}{\langle k_3 \rangle^2} \mathd s_1 \mathd s_2 \mathd s_3 \mathd t
  \end{array}\]}
  Introduce the notation $\mathbbm{Z}_+^3 = \{ n \in \mathbbm{Z}^3 : n = (n_1,
  n_2, n_3) \tmop{with} n_i \geqslant 0 \}$. After restricting the sum to
  $(\mathbbm{Z}_+^3)^3$ we get the bound
   {\footnotesize \[
  \begin{array}{ll}
       \geqslant & \int^{T / 2}_{T / 8} \sum_k \frac{\sigma^2_t (k)}{\langle k
       \rangle^2} \sum_{\tmscript{\begin{array}{l}
         k_1, k_2, k_3 \in \mathbbm{Z}_+^3\\
         k_1 + k_2 + k_3 = k
       \end{array}}} \int^{T / 8}_{3 T / 32} \int^{s_1}_{3 T / 32}
       \int^{s_2}_{3 T / 32} \frac{\sigma^2_{s_1} (k_1)}{\langle k_1
       \rangle^2} \frac{\sigma^2_{s_2} (k_2)}{\langle k_2 \rangle^2}
       \frac{\sigma^2_{s_3} (k_3)}{\langle k_3 \rangle^2} \mathd s_1 \mathd
       s_2 \mathd s_3 \mathd t\\
       \gtrsim & \frac{1}{T^2} \sum_{k \in \mathbbm{Z}_+^3} (\rho_{T / 2} (k)
       - \rho_{T / 8} (k)) \sum_{\tmscript{\begin{array}{l}
         k_1, k_2, k_3 \in \mathbbm{Z}_+^3\\
         k_1 + k_2 + k_3 = k
       \end{array}}} \int^{T / 8}_{3 T / 32} \int^{s_1}_{3 T / 32}
       \int^{s_2}_{3 T / 32} \frac{\sigma^2_{s_1} (k_1)}{\langle k_1
       \rangle^2} \frac{\sigma^2_{s_2} (k_2)}{\langle k_2 \rangle^2}
       \frac{\sigma^2_{s_3} (k_3)}{\langle k_3 \rangle^2} \mathd s_1 \mathd
       s_2 \mathd s_3
     \end{array} \]}
  Now, for large enough $T$ if $k_1 + k_2 + k_3 = k$ and $\langle k_i \rangle
  \leqslant T / 8$ then $\langle k \rangle \leqslant T / 2 \times 0.9$. \
  Furthermore if $T$ large enough and $k_1, k_2, k_3 \in \mathbbm{Z}_+^3$ and
  $k_1 + k_2 + k_3 = k$, while $\langle k_i \rangle \geqslant (3 T / 32)
  \times 0.9$ (recall that if $\langle k_i \rangle < (3 T / 32) \times
  0.9$\quad and $s > 3 T / 32$ then $\sigma_s (k_1) = 0$) we have \ $\langle k
  \rangle \geqslant T / 8$. So for any $k$ for which the integral is nonzero
  we have \ $\rho_{T / 2} (k) - \rho_{T / 8} (k) = 1$ (recall that $\rho = 1$
  on $B (0, 9 / 10)$ and $\rho = 0$ outside of $B (0, 1)$). This implies
  {\footnotesize
  \begin{eqnarray*}
    &  & \frac{1}{T^2} \sum_{k \in \mathbbm{Z}_+^3} (\rho_{T / 2} (k) -
    \rho_{T / 8} (k)) \sum_{\tmscript{\begin{array}{l}
      k_1, k_2, k_3 \in \mathbbm{Z}_+^3\\
      k_1 + k_2 + k_3 = k
    \end{array}}} \int^{T / 8}_{3 T / 32} \int^{s_1}_{3 T / 32} \int^{s_2}_{3
    T / 32} \frac{\sigma^2_{s_1} (k_1)}{\langle k_1 \rangle^2}
    \frac{\sigma^2_{s_2} (k_2)}{\langle k_2 \rangle^2} \frac{\sigma^2_{s_3}
    (k_3)}{\langle k_3 \rangle^2} \mathd s_1 \mathd s_2 \mathd s_3\\
    & = & \frac{1}{T^2} \sum_{\tmscript{\begin{array}{l}
      k_1, k_2, k_3 \in \mathbbm{Z}_+^3
    \end{array}}} \int^{T / 8}_{3 T / 32} \int^{s_1}_{3 T / 32} \int^{s_2}_{3
    T / 32} \frac{\sigma^2_{s_1} (k_1)}{\langle k_1 \rangle^2}
    \frac{\sigma^2_{s_2} (k_2)}{\langle k_2 \rangle^2} \frac{\sigma^2_{s_3}
    (k_3)}{\langle k_3 \rangle^2} \mathd s_1 \mathd s_2 \mathd s_3\\
    & \gtrsim & T
  \end{eqnarray*}}
  \end{proof}

Next we upgrade this bound to almost sure divergence.

\begin{lemma}
  There exists a $\delta_0 > 0$ such that for any $\delta_0 \geqslant \delta >
  0$, ,there exists a sequence $(T_n)_n$ such that $\mathbbm{P}- \tmop{almost}
  \tmop{surely}$
  \[ \frac{1}{T_n^{1 - \delta}} \int^{\infty}_0 \int_{\Lambda} \left(
     \theta_{T_n} J_t \mathbbm{W}_t^{\theta_{T_n}, 3} \right) J_t
     \mathbbm{W}_t^3 \mathd t \rightarrow \infty . \]
\end{lemma}

\begin{proof}
  Define
  \[ G_T \assign \frac{1}{T^{1 - \delta}} \int^{\infty}_0 \int_{\Lambda}
     (\theta_T J_t \mathbbm{W}_t^{\theta_T, 3}) J_t \mathbbm{W}_t^3 \mathd t +
     \sup_{t < \infty} \| W_t \|^K_{\VV^{- 1 / 2 - \varepsilon}} . \]
  We will show that $e^{- G_T} \rightarrow 0$ in $L^1 (\mathbbm{P})$, which
  implies that there exists a subsequence $T_n$ such that $e^{- G_{T_n}}
  \rightarrow 0$ almost surely. From this our statement follows. By the
  Bou{\'e}--Dupuis formula
  \begin{eqnarray*}
    &  & \\
    - \log \mathbbm{E} [e^{- G_T}] & = & \inf_{v \in \mathbbm{H}_a}
    \mathbbm{E} \left[ \frac{1}{T^{1 - \delta}} 16 \int^{\infty}_0
    \int_{\Lambda} (\theta_T J_t \llbracket \theta_T ((W_t + I_t (v))^3)
    \rrbracket) J_t \llbracket (W_t + I_t (v))^3 \rrbracket \mathd t +
    \right.\\
    &  & \qquad \qquad \left. + \sup_{t < \infty} \| W_t + I_t (v)
    \|^K_{\VV^{- 1 / 2 - \varepsilon}} + \frac{1}{2} \int^{\infty}_0 \| v_t
    \|^2_{L^2} \mathd t \right]\\
    & = & \inf_{v \in \mathbbm{H}_a} \mathbbm{E} \left[ \frac{1}{T^{1 -
    \delta}} \int^{\infty}_0 \int_{\Lambda} (\theta_T J_t
    \mathbbm{W}_t^{\theta_T, 3}) J_t \mathbbm{W}_t^3 \mathd t \right. +\\
    &  & \qquad \qquad + \frac{1}{T^{1 - \delta}} \sum_{(i, j) \in \{ 0, 1,
    2, 3 \}^2 \setminus (0, 0)} \int^T_0 \int_{\Lambda} A_t^i B_t^j \mathd t\\
    &  & \qquad \qquad \left. + \sup_{t < \infty} \| W_t + I_t (v)
    \|^K_{\VV^{- 1 / 2 - \varepsilon}} + \frac{1}{2} \int^{\infty}_0 \| v_t
    \|^2_{L^2} \mathd t \right]\\
    & \geqslant & \inf_{v \in \mathbbm{H}_a} \mathbbm{E} \left[ \frac{1}{T^{1
    - \delta}} \int^{\infty}_0 \int_{\Lambda} (\theta_T J_t
    \mathbbm{W}_t^{\theta_T, 3}) J_t \mathbbm{W}_t^3 \mathd t \right.\\
    &  & \qquad \qquad + \frac{1}{T^{1 - \delta}} \sum_{(i, j) \in \{ 0, 1,
    2, 3 \}^2 \setminus (0, 0)} \int^T_0 \int_{\Lambda} A_t^i B_t^j \mathd t\\
    &  & \qquad \qquad \left. + \frac{1}{2} \sup_{t < \infty} \| I_t (v)
    \|^K_{\VV^{- 1 / 2 - \varepsilon}} - C \sup_{t < \infty} \| W_t
    \|^K_{\VV^{- 1 / 2 - \varepsilon}} + \frac{1}{2} \int^{\infty}_0 \| v_t
    \|^2_{L^2} \mathd t \right]
  \end{eqnarray*}
  where where have used that $\theta_T J_t = 0$ for $t \geqslant T$ and
  introduced the notations, for $0 \leqslant i \leqslant 3$,
  \[ A_t^i \assign 4 \binom{3}{i} J_t \theta_T (\llbracket (\theta_T W_t)^{3 -
     i} \rrbracket (\theta_T I_t (v))^i), \]
  and
  \[ B^i_t \assign 4 \binom{3}{i} J_t (\llbracket W_t^{3 - i} \rrbracket (I_t
     (v))^i) . \]
  Our aim now to prove that the last three terms are bounded below uniformly
  as $T \rightarrow \infty$ (while we already know that the first one
  diverges). For $i \in \{ 1, 2, 3 \}$
  \[ \| A_t^i \|^2_{L^2} + \| B_t^i \|^2_{L^2} \lesssim \langle t \rangle^{- 1
     + \delta} \left( \| I_t (u) \|^K_{\VV^{- 1 / 2 - \varepsilon}} + \| I_t
     (u) \|^2_{H^1} + Q_t (W) \right) \]
  by Lemmas~\ref{analyticestimate2} and~\ref{analyticestimate4}. Here $Q_t
  (W)$ is a random variable only depending on $W$ such that $\sup_t
  \mathbbm{E} [| Q_t (W) |^p] < \infty$ for any $p < \infty$. Then
  \[ \begin{array}{ll}
       & \frac{1}{T^{1 - \delta}} \sum_{(i, j) \in \{ 0, 1, 2, 3 \}^2
       \setminus (0, 0)} \int^T_0 \int_{\Lambda} | A_t^i B_t^j | \mathd t\\
       \leqslant & \frac{1}{T^{1 - \delta}} \sum_{(i, j) \in \{ 1, 2, 3 \}^2}
       \int^T_0 \| A_t^i \|^2_{L^2} + \| B_t^j \|^2_{L^2} \mathd t\\
       & + \frac{1}{T^{1 - \delta}} \sum_{i \in \{ 1, 2, 3 \}} \int^T_0 \|
       A_t^0 \|_{L^2}  \| B_t^i \|_{L^2} \mathd t + \frac{1}{T^{1 - \delta}}
       \sum_{i \in \{ 1, 2, 3 \}} \int^T_0 \| A_t^i \|_{L^2}  \| B_t^0
       \|_{L^2} \mathd t.
     \end{array} \]
  Now for the first term we obtain
  \[ \begin{array}{ll}
       & \mathbbm{E} \left[ \frac{1}{T^{1 - \delta}} \sum_{(i, j) \in \{ 1,
       2, 3 \}^2} \int^T_0 \| A_t^i \|^2_{L^2} + \| B_t^j \|^2_{L^2} \mathd t
       \right]\\
       = & \mathbbm{E} \left[ \frac{1}{T^{1 - \delta}} \sum_{(i, j) \in \{ 1,
       2, 3 \}^2} \int^T_0 \langle t \rangle^{- 1 + \delta} \left( \| I_t (v)
       \|^K_{\VV^{- 1 / 2 - \varepsilon}} + \| I_t (v) \|^2_{H^1} + Q_t (W)
       \right) \mathd t \right]\\
       = & \frac{C}{T^{1 - 2 \delta}} \mathbbm{E} \left[ \sup_t \left( \| I_t
       (v) \|^K_{\VV^{- 1 / 2 - \varepsilon}} + \| I_t (v) \|^2_{H^1} \right)
       \right] + \frac{C}{T^{1 - 2 \delta}} .
     \end{array} \]
  For the second term we use that $\| A_t^0 \|_{L^2} \leqslant Q_t (W)$ so
  \begin{eqnarray*}
    &  & \frac{1}{T^{1 - \delta}} \mathbbm{E} \left[ \int^T_0 \| A_t^0
    \|_{L^2}  \| B_t^i \|_{L^2} \mathd t \right]\\
    & \leqslant & \frac{1}{T^{1 - \delta}} \mathbbm{E} \left[ \int^T_0
    \langle t \rangle^{- 1 / 2} \| A_t^0 \|^2_{L^2} \mathd t + \int^T_0
    \langle t \rangle^{1 / 2} \| B_t^i \|^2_{L^2} \mathd t \right]\\
    & \lesssim & \frac{1}{T^{1 - \delta}} \mathbbm{E} \left[ \int^T_0 \langle
    t \rangle^{- 1 / 2} \| A_t^0 \|^2_{L^2} \mathd t \right]\\
    &  & + \frac{1}{T^{1 - \delta}} \mathbbm{E} \left[ \int^T_0 \langle t
    \rangle^{- 1 / 2 + \delta} \left( \| I_t (v) \|^K_{\VV^{- 1 / 2 -
    \varepsilon}} + \| I_t (v) \|^2_{H^1} + Q_t (W) \right) \mathd t \right]\\
    & \lesssim & \frac{C}{T^{1 / 2 - 2 \delta}} \mathbbm{E} \left[ \sup_t
    \left( \| I_t (v) \|^K_{\VV^{- 1 / 2 - \varepsilon}} + \| I_t (v)
    \|^2_{H^1} \right) \right] + \frac{C}{T^{1 / 2 - 2 \delta}}
  \end{eqnarray*}
  Since $\sup_t \| I_t (v) \|^2_{H^1} \lesssim \int^{\infty}_0 \| v_t
  \|^2_{L^2} \mathd t$ in total we obtain for $T$ large enough. The third term
  is estimated analogously.
  \begin{eqnarray*}
    &  & - \log \mathbbm{E} [e^{- G_T}]\\
    & \geqslant & \inf_{v \in \mathbbm{H}_a} \mathbbm{E} \left[ \frac{1}{T^{1
    - \delta}} \int^{\infty}_0 \int_{\Lambda} (\theta_T J_t
    \mathbbm{W}_t^{\theta_T, 3}) J_t \mathbbm{W}_t^3 \mathd t + \left(
    \frac{1}{2} - \frac{C}{T^{1 / 2 - 2 \delta}} \right) \sup_{t < \infty} \|
    I_t (v) \|^K_{\VV^{- 1 / 2 - \varepsilon}} \right.\\
    &  & \qquad \left. - C \sup_{t < \infty} \| W_t \|^K_{\VV^{- 1 / 2 -
    \varepsilon}} + \left( \frac{1}{2} - \frac{C}{T^{1 / 2 - 2 \delta}}
    \right) \int^{\infty}_0 \| v_t \|^2_{L^2} \mathd t - \frac{C}{T^{1 / 2 - 2
    \delta}} \right]\\
    & \geqslant & \mathbbm{E} \left[ \frac{1}{T^{1 - \delta}} \int^{\infty}_0
    \int_{\Lambda} (\theta_T J_t \mathbbm{W}_t^{\theta_T, 3}) J_t
    \mathbbm{W}_t^3 \mathd t \right] - C \rightarrow \infty
  \end{eqnarray*}
  
\end{proof}

Next we prove an estimate which will help with the proof of the main theorem.

\begin{lemma}
  \label{Qcubeestimate}We have
  \[ \sup_T \mathbbm{E}_{\mathbbm{Q}^u} \left[ \int^{\infty}_0 \int_{\Lambda}
     \frac{1}{t^{1 + \delta}} (\theta_T J_t \mathbbm{W}_t^{\theta_T, 3})^2
     \mathd t \right] < \infty . \]
  Furthermore, there exists a (deterministic) subsequence $(T_n)_n$ such that
  \
  \[ \frac{1}{T_n^{1 / 2 + \delta}} \left| \int^{\infty}_0 \int_{\Lambda}
     \theta_{T_n} J_t \mathbbm{W}_t^{\theta_{T_n}, 3} \mathd X^u_t \right|
     \rightarrow 0 \]
  $\mathbbm{Q}^u$ almost surely.
\end{lemma}

\begin{proof}
  Recall that under $\mathbbm{Q}^u$ we have $W_t = W_t^u + I_t (u)$ where $u$
  is defined above by (\ref{eq:drift}) and $\tmop{Law}_{\mathbbm{Q}^u} (W^u) =
  \tmop{Law}_{\mathbbm{P}} (W)$. With this in mind we compute
  \[ \int^T_0 \int_{\Lambda} \frac{1}{t^{1 + \delta}} (\theta_T J_t
     \mathbbm{W}_t^{\theta_T, 3})^2 \mathd t = \sum_{i, j \leqslant 3}
     \int^T_0 \int_{\Lambda} \frac{1}{t^{1 + \delta}} A_t^i A_t^j \mathd t, \]
  where, as above,
  \[ A_t^i = 4 \binom{3}{i} J_t \theta_T (\llbracket (\theta_T W^u_t)^{3 - i}
     \rrbracket (\theta_T I_t (u))^i) . \]
  By Lemmas~\ref{analyticestimate2} and~\ref{analyticestimate4} we have that
  $\mathbbm{E}_{\mathbbm{Q}^u} [\| A_t^i \|^2_{L^2}] \leqslant C$ so the
  Cauchy--Schwartz inequality gives the result.
\end{proof}

\begin{theorem}
  There exists a sequence $(T_n)_n$ such that, $\mathbbm{Q}^u$ almost surely,
  \[ \frac{1}{T^{1 - \delta}_n} \int_{\Lambda} \llbracket (\theta_{T_n}
     W_{\infty})^4 \rrbracket \rightarrow - \infty . \]
\end{theorem}

\begin{proof}
  We have
  \[ \int_{\Lambda} \llbracket (\theta_T W_{\infty})^4 \rrbracket =
     \int^{\infty}_0 \int_{\Lambda} \theta_T J_t \mathbbm{W}_t^{\theta_T, 3}
     \mathd X_t . \]
  Now since $\mathd X_t = \mathd X_t^u + u_t \mathd t$ we have
  \[ \begin{array}{ll}
       & \frac{1}{T^{1 - \delta}} \int^{\infty}_0 \int_{\Lambda} \theta_T J_t
       \mathbbm{W}_t^{\theta_T, 3} \mathd X_t\\
       = & \frac{1}{T^{1 - \delta}} \int^{\infty}_0 \int_{\Lambda} \theta_T
       J_t \mathbbm{W}_t^{\theta_T, 3} \mathd X^u_t + \frac{1}{T^{1 - \delta}}
       \int^{\infty}_0 \int_{\Lambda} \theta_T J_t \mathbbm{W}_t^{\theta_T, 3}
       u_t \mathd t\\
       = & \frac{1}{T^{1 - \delta}} \int^{\infty}_0 \int_{\Lambda} \theta_T
       J_t \mathbbm{W}_t^{\theta_T, 3} \mathd X^u_t - \frac{\lambda}{T^{1 -
       \delta}} \int^{\infty}_0 \int_{\Lambda} \theta_T J_t
       \mathbbm{W}_t^{\theta_T, 3} J_t \mathbbm{W}_t^{u, 3} \mathd t\\
       & - \frac{\lambda}{T^{1 - \delta}} \int^{\infty}_{\bar{T}}
       \int_{\Lambda} \theta_T J_t \mathbbm{W}_t^{\theta_T, 3} J_t
       (\mathbbm{W}_t^{u, 2} \succ I^{\flat}_t (u)) \mathd t\\
       & - \frac{1}{T^{1 - \delta}} \int^{\infty}_0 \int_{\Lambda} \theta_T
       J_t \mathbbm{W}_t^{\theta_T, 3} J_t \langle D \rangle^{- 1 / 2}
       \llbracket (\langle D \rangle^{- 1 / 2} W^u_t)^n \rrbracket \mathd t.
     \end{array} \]
  The first term goes to $0$ $\mathbbm{Q}^u$-almost surely by
  Lemma~\ref{Qcubeestimate}. To analyze the third term we estimate
  \begin{equation}
    \begin{array}{ll}
      & \frac{1}{T^{1 - \delta}} \int^{\infty}_{\bar{T}} \int_{\Lambda}
      \theta_T J_t \mathbbm{W}_t^{\theta_T, 3} J_t (\mathbbm{W}_t^{u, 2} \succ
      I^{\flat}_t (u)) \mathd t\\
      = & \frac{1}{T^{1 - \delta}} \int^T_{\bar{T}} \int_{\Lambda} \theta_T
      J_t \mathbbm{W}_t^{\theta_T, 3} J_t (\mathbbm{W}_t^{u, 2} \succ
      I^{\flat}_t (u)) \mathd t\\
      \leqslant & \frac{1}{T^{1 - \delta}} \int^T_{\bar{T}} \| \theta_T J_t
      \mathbbm{W}_t^{\theta_T, 3} \|_{L^2} \| J_t (\mathbbm{W}_t^{u, 2} \succ
      I^{\flat}_t (u)) \|_{L^2} \mathd t\\
      \lesssim & \frac{1}{T^{1 - \delta}} \int^T_{\bar{T}} t^{- 1 / 2 + \delta
      / 2} \| \theta_T J_t \mathbbm{W}_t^{\theta_T, 3} \|_{L^2}  \|
      \mathbbm{W}_t^{u, 2} \|_{\mathcal{C}^{- 1 - \delta / 2}} \| I_t (u)
      \|_{L^2} \mathd t\\
      \leqslant & T^{- 1 / 2 - 2 \delta} \left( \int^T_{\bar{T}} \| \theta_T
      J_t \mathbbm{W}_t^{\theta_T, 3} \|^2_{L^2} \mathd t \right)^{1 / 2}\\
      & \qquad \times T^{- 1 / 2 + 2 \delta} \left( \int^T_{\bar{T}} t^{- 1 +
      \delta} (\| \mathbbm{W}_t^{u, 2} \|_{\mathcal{C}^{- 1 - \delta / 2}} \|
      I_t (u) \|_{L^2})^2 \right)^{1 / 2}
    \end{array} \label{bcross1}
  \end{equation}
  By the computation from Lemma~\ref{Qcubeestimate} we have
  \[ \mathbbm{E}_{\mathbbm{Q}^u} \left[ T^{- 1 / 2 - 2 \delta} \left(
     \int^T_{\bar{T}} \| \theta_T J_t \mathbbm{W}_t^{\theta_T, 3} \|^2_{L^2}
     \mathd t \right)^{1 / 2} \right] \rightarrow 0, \]
  and $\sup_t \mathbbm{E}_{\mathbbm{Q}^u} [(\| \mathbbm{W}_t^{u, 2}
  \|_{\mathcal{C}^{- 1 - \delta / 2}} \| I_t (u) \|_{L^2})^2] < \infty$, so
  (\ref{bcross1}) converges to $0$ in $L^1 (\mathbbm{Q}^u)$. For the fourth
  term we proceed in the same way:
  \begin{eqnarray*}
    &  & \int^{\infty}_0 \int_{\Lambda} \theta_T J_t \mathbbm{W}_t^{\theta_T,
    3} J_t \langle D \rangle^{- 1 / 2} \llbracket (\langle D \rangle^{- 1 / 2}
    W^u_t)^n \rrbracket \mathd t\\
    & = & \int^T_0 \int_{\Lambda} \theta_T J_t \mathbbm{W}_t^{\theta_T, 3}
    J_t \langle D \rangle^{- 1 / 2} \llbracket (\langle D \rangle^{- 1 / 2}
    W^u_t)^n \rrbracket \mathd t\\
    & \leqslant & \int^T_0 \| \theta_T J_t \mathbbm{W}_t^{\theta_T, 3}
    \|_{L^2} \| J_t \langle D \rangle^{- 1 / 2} \llbracket (\langle D
    \rangle^{- 1 / 2} W^u_t)^n \rrbracket \|_{L^2} \mathd t\\
    & \lesssim & \int^T_0 (\| \theta_T J_t \mathbbm{W}_t^{\theta_T, 3}
    \|_{L^2}) t^{- 2 + \delta} \| \llbracket (\langle D \rangle^{- 1 / 2}
    W^u_t)^n \rrbracket \|_{H^{- \delta}} \mathd t\\
    & \leqslant & \left( \int^T_0 t^{- 2 (1 - \delta)} (\| \theta_T J_t
    \mathbbm{W}_t^{\theta_T, 3} \|_{L^2})^2 \right)^{1 / 2} \left( \int^T_0
    t^{- 2 (1 - \delta)} \| \llbracket (\langle D \rangle^{- 1 / 2} W^u_t)^n
    \rrbracket \|^2_{H^{- \delta}} \right)^{1 / 2}
  \end{eqnarray*}
  which is bounded in expectation uniformly in $T$, so the fourth term goes to
  $0$ in $L^1 (\mathbbm{Q}^u)$ as well. It remains to analyze the second term.
  Again introducing the notation
  \[ A_t^i = 4 \binom{3}{i} J_t \theta_T (\llbracket (\theta_T W^u_t)^{3 - i}
     \rrbracket (\theta_T I_t (u))^i), \]
  \[ \mathbbm{W}_t^{\theta_T, u, 3} = 4 \llbracket (\theta_T W^u_t)^3
     \rrbracket, \]
  we have
  \begin{eqnarray*}
    &  & \frac{1}{T^{1 - \delta}} \int^{\infty}_0 \int_{\Lambda} \theta_T J_t
    \mathbbm{W}_t^{\theta_T, 3} J_t \mathbbm{W}_t^{u, 3} \mathd t\\
    & = & \frac{1}{T^{1 - \delta}} \int^T_0 \int_{\Lambda} \theta_T J_t
    \mathbbm{W}_t^{\theta_T, u, 3} J_t \mathbbm{W}_t^{u, 3} \mathd t + \sum_{1
    \leqslant i \leqslant 3} \frac{1}{T^{1 - \delta}} \int^T_0 \int_{\Lambda}
    A_t^i J_t \mathbbm{W}_t^{u, 3} \mathd t.
  \end{eqnarray*}
  Now observe that
  \[ \frac{1}{T^{1 - \delta}} \int^T_0 \int_{\Lambda} \theta_T J_t
     \mathbbm{W}_t^{\theta_T, u, 3} J_t \mathbbm{W}_t^{u, 3} \mathd t_{\mathbbm{Q}^u} \sim_{\mathbbm{P} \quad} \frac{1}{T^{1 - \delta}}
     \int^T_0 \int_{\Lambda} \theta_T J_t \mathbbm{W}_t^{\theta_T, 3} J_t
     \mathbbm{W}_t^3 \mathd t_{}, \]
  so the $\limsup$of this is $\infty$ almost surely. To estimate the sum we
  again observe that for $i \geqslant 3$ $\mathbbm{E}_{\mathbbm{Q}^u} [\|
  A_t^i \|^2_{L^2}] \lesssim \langle t \rangle^{- 1 + \delta}$ and by Young's
  inequality
  \[ \begin{array}{lll}
       \int^T_0 \int_{\Lambda} A_t^i J_t \mathbbm{W}_t^{u, 3} \mathd t &
       \leqslant & \int^T_0 \int_{\Lambda} \| A_t^i \|_{L^2} \| J_t
       \mathbbm{W}_t^{u, 3} \|_{L^2} \mathd t\\
       & \leqslant & \int^T_0 \int_{\Lambda} \langle t \rangle^{1 / 3} \|
       A_t^i \|_{L^2} \langle t \rangle^{- 1 / 3} \| J_t \mathbbm{W}_t^{u, 3}
       \|_{L^2} \mathd t\\
       & \leqslant & \int^T_0 \int_{\Lambda} \langle t \rangle^{2 / 3} \|
       A_t^i \|^2_{L^2} + \int^T_0 \int_{\Lambda} \langle t \rangle^{- 2 / 3}
       \| J_t \mathbbm{W}_t^{u, 3} \|^2_{L^2} \mathd t.
     \end{array} \]
  Taking expectation we obtain
  \begin{eqnarray*}
    &  & \frac{1}{T^{1 - \delta}} \mathbbm{E} \left[ \int^T_0 \int_{\Lambda}
    A_t^i J_t \mathbbm{W}_t^{u, 3} \mathd t \right]\\
    & \leqslant & \frac{1}{T^{1 - \delta}} \mathbbm{E} \left[ \int^T_0
    \int_{\Lambda} \langle t \rangle^{2 / 3} \| A_t^i \|^2_{L^2} \right] +
    \frac{1}{T^{1 - \delta}} \mathbbm{E} \left[ \int^T_0 \int_{\Lambda}
    \langle t \rangle^{- 2 / 3} \| J_t \mathbbm{W}_t^{u, 3} \|^2_{L^2} \mathd
    t \right]\\
    & \lesssim & \frac{1}{T^{1 - \delta}} \int^T_0 \int_{\Lambda} \langle t
    \rangle^{- 1 / 3 + \delta} + \frac{1}{T^{1 - \delta}} \int^T_0
    \int_{\Lambda} \langle t \rangle^{- 2 / 3} \mathd t \rightarrow 0
  \end{eqnarray*}
  We have deduced that
  \[ \frac{1}{T^{1 - \delta}} \int_{\Lambda} \llbracket (\theta_T
     W_{\infty})^4 \rrbracket = - \frac{1}{T^{1 - \delta}} \int^T_0
     \int_{\Lambda} \theta_T J_t \mathbbm{W}_t^{\theta_T, u, 3} J_t
     \mathbbm{W}_t^{u, 3} \mathd t_{} + R_T, \]
  where $R_T \rightarrow 0$ in $L^1 (\mathbbm{Q}^u)$. We can conclude by
  selecting a subsequence $(T_n)_n$ such that
  \[ \frac{1}{T_n^{1 - \delta}} \int^{T_n}_0 \int_{\Lambda} \theta_T J_t
     \mathbbm{W}_t^{\theta_{T_n}, u, 3} J_t \mathbbm{W}_t^{u, 3} \mathd t
     \rightarrow \infty \]
  $\mathbbm{Q}^u$-almost surely and $R_{T_n} \rightarrow 0$,
  $\mathbbm{Q}^u$-almost surely.
\end{proof}

\section{Some analytic estimates }\label{sec:analytic}

We collect in this final section various technical estimates needed to
complete the proof of Lemma~\ref{lemma:estimateremainder}.

\begin{proposition}
  \label{FractionalLeibniz}Let $1 < p < \infty$ and $p_1, p_2, p_1', p_2' > 1$
  such that $\frac{1}{p_1} + \frac{1}{p_2} = \frac{1}{p_1'} + \frac{1}{p_2'} =
  \frac{1}{p}$. Then for every $s, \alpha \geqslant 0$ \
  \[ \| \langle \mathD \rangle^s (f g) \|_{L^p} \lesssim \| \langle \mathD
     \rangle^{s + \alpha} f \|_{L^{p_2}}  \| \langle \mathD \rangle^{- \alpha}
     g \|_{L^{p_1}} + \| \langle \mathD \rangle^{s + \alpha} g \|_{L^{p_1'}} 
     \| \langle \mathD \rangle^{- \alpha} f \|_{L^{p_2'}} . \]
\end{proposition}

\begin{proof}
  See~{\cite{gulisashvili_exact_1996}}.
\end{proof}

\begin{lemma}
  There exists $\varepsilon > 0, n \in \mathbbm{N}$ such that for any $\delta
  > 0$ there exists $C_{\delta} < \infty$ for which the following inequality
  holds for any $\phi \in H^1 (\Lambda)$ \label{L4estimate}
  \[ \| \phi \|^{4 + \varepsilon}_{L^4} \leqslant C \| \phi \|^{n + 1}_{W^{- 1
     / 2, n + 1}} + \delta \| \phi \|^2_{H^1} + C_{\delta} . \]
\end{lemma}

\begin{proof}
  
  \begin{eqnarray*}
    \int \phi^4 \mathd x & \leqslant & \| \langle \mathD \rangle^{- 1 / 2}
    \phi \|_{L^8} \| \langle \mathD \rangle^{1 / 2} \phi^3 \|_{L^{8 / 7}}\\
    & \leqslant & \| \langle \mathD \rangle^{- 1 / 2} \phi \|_{L^8} \|
    \langle \mathD \rangle^{1 / 2} \phi \|_{L^{8 / 3}} \| \phi \|^2_{L^4}\\
    & \leqslant & \| \langle \mathD \rangle^{- 1 / 2} \phi \|_{L^8} \| \phi
    \|^{1 / 2}_{H^1} \| \phi \|^{5 / 2}_{L^4}
  \end{eqnarray*}
  So
  \begin{eqnarray*}
    (\| \phi \|^4_{L^4})^{21 / 20} & \leqslant & \| \langle \mathD \rangle^{-
    1 / 2} \phi \|^{21 / 20}_{L^8} \| \phi \|^{21 / 40}_{H^1} \| \phi \|^{104
    / 40}_{L^4}
  \end{eqnarray*}
  and applying Young's inequality with the exponents $(32, 32 / 9, 32 / 22)$,
  we obtain
  \[ \begin{array}{lll}
       \| \langle \mathD \rangle^{- 1 / 2} \phi \|^{21 / 20}_{L^8} \| \phi
       \|^{21 / 40}_{H^1} \| \phi \|^{104 / 40}_{L^4} & \leqslant & C_{\delta}
       \| \langle \mathD \rangle^{- 1 / 2} \phi \|^{168 / 5}_{L^8} + \delta \|
       \phi \|^{16 / 9}_{H^1} + \delta \| \phi \|^{208 / 55}_{L^4}\\
       & \leqslant & \| \langle \mathD \rangle^{- 1 / 2} \phi \|^{34}_{L^8} +
       \delta \| \phi \|^2_{H^1} + \delta (\| \phi \|^4_{L^4})^{21 / 20} +
       C_{\delta}
     \end{array} \]
  and subtracting $\delta (\| \phi \|^4_{L^4})^{21 / 20}$ on both sides of the
  inequality gives the result. 
\end{proof}

\begin{lemma}
  The following estimates hold with $\varepsilon > 0$ small
  enough\label{analyticestimate1}
  \[ \| J_t (\llbracket W_t^2 \rrbracket \succ (1 - \theta_t) I_t (w))
     \|^2_{L^2} \lesssim \frac{1}{\langle t \rangle^{1 + \varepsilon}} \left(
     \int^t_0 \| w_s \|^2 \mathd s + \| I_t (w) \|^n_{W^{- 1 / 2, n + 1}} + \|
     \llbracket W_t^2 \rrbracket \|^n_{\VV^{- 1 - \varepsilon}} \right) \]
  \[ \| J_t (\llbracket W_t^2 \rrbracket \circ I_t (w)) \|^2_{L^2} \lesssim
     \frac{1}{\langle t \rangle^{1 + \varepsilon}} \left( \int^t_0 \| w_s \|^2
     \mathd s + \| I_t (w) \|^n_{W^{- 1 / 2, n + 1}} + \| \llbracket W_t^2
     \rrbracket \|^n_{\VV^{- 1 - \varepsilon}} \right) \]
  \[ \| J_t \llbracket W_t^2 \rrbracket \prec I_t (w) \|^2_{L^2} \lesssim
     \frac{1}{\langle t \rangle^{1 + \varepsilon}} \left( \int^t_0 \| w_s \|^2
     \mathd s + \| I_t (w) \|^n_{W^{- 1 / 2, n + 1}} + \| \llbracket W_t^2
     \rrbracket \|^n_{\VV^{- 1 - \varepsilon}} \right) \]
\end{lemma}

\begin{proof}
  We observe that since $\llbracket W_t^2 \rrbracket$ is spectrally supported
  in a ball or radius $\thicksim t$
  \[ \| \llbracket W_t^2 \rrbracket \|_{\VV^{- 1 + \varepsilon}} \lesssim
     \langle t \rangle^{2 \varepsilon} \| \llbracket W_t^2 \rrbracket
     \|_{\VV^{- 1 - \varepsilon}} . \]
  For the first estimate we know that $(1 - \theta_t) I_t (w)$ is supported in
  an annulus of radius $\sim t$, so $\| (1 - \theta_t) I_t (w) \|_{L^2}
  \lesssim \langle t \rangle^{- 1 + \varepsilon} \| I_t (w) \|_{H^{1 -
  \varepsilon}}$ and furthermore by interpolation $\| I_t (w) \|_{H^{1 -
  \varepsilon}} \lesssim \| I_t (w) \|^{1 - \varepsilon}_{H^1} \| I_t (w)
  \|^{\varepsilon}_{L^2} \lesssim \| I_t (w) \|^{1 - \varepsilon}_{H^1} \| I_t
  (w) \|^{\varepsilon}_{L^4}$. By definition $\langle t \rangle^{1 / 2} J_t$
  is a uniformly bounded Fourier multiplier regularizing by 1, and putting
  everything together, by paraproduct estimates
  \begin{eqnarray*}
    \| J_t (\llbracket W_t^2 \rrbracket \succ (1 - \theta_t) I_t (w))
    \|^2_{L^2} & \lesssim & \langle t \rangle^{- 1} \langle t \rangle^{2
    \varepsilon} \langle t \rangle^{- 2 + 2 \varepsilon} \| \llbracket W_t^2
    \rrbracket \|^2_{\VV^{- 1 - \varepsilon}} \| I_t (w) \|^2_{H^{1 -
    \varepsilon}}\\
    & \lesssim & \langle t \rangle^{- 1} \langle t \rangle^{2 \varepsilon}
    \langle t \rangle^{- 2 + 2 \varepsilon} \| \llbracket W_t^2 \rrbracket
    \|^2_{\VV^{- 1 - \varepsilon}} \| I_t (w) \|^{2 - 2 \varepsilon}_{H^1} \|
    I_t (w) \|^{2 \varepsilon}_{L^4}\\
    \small{(\varepsilon = 2 / 7)} \qquad & \lesssim & \langle t \rangle^{- 3 /
    2} \left( \| \llbracket W_t^2 \rrbracket \|^{14}_{\VV^{- 1 - \varepsilon}}
    + \| I_t (w) \|^2_{H^1} + \| I_t (w) \|^4_{L^4} \right)\\
    & \lesssim & \langle t \rangle^{- 3 / 2} \left( \int^t_0 \| w \|^2 \mathd
    s + \| I_t (w) \|^n_{W^{- 1 / 2, n + 1}} + \| \llbracket W_t^2 \rrbracket
    \|^{14}_{\VV^{- 1 - \varepsilon}} \right)
  \end{eqnarray*}
  For the second term in addition observe that the function $\langle t
  \rangle^{1 / 2} J_t$ is spectrally supported in an annulus of radius~$\sim
  t$, and regularizes by $1$ so again by estimates for the resonant product
  \[ \begin{array}{lll}
       \| J_t (\llbracket W_t^2 \rrbracket \circ I_t (w)) \|^2_{L^2} &
       \lesssim & \langle t \rangle^{- 3} \| \llbracket W_t^2 \rrbracket
       \|^2_{\VV^{- 1 + 2 \varepsilon}} \| I_t (w) \|^2_{H^{1 -
       \varepsilon}}\\
       & \lesssim & \langle t \rangle^{- 3} \langle t \rangle^{6 \varepsilon}
       \| \llbracket W_t^2 \rrbracket \|^2_{\VV^{- 1 - \varepsilon}} \| I_t
       (w) \|^2_{H^{1 - \varepsilon}}
     \end{array} \]
  For the third estimate again applying paraproduct estimates and the
  properties of $J$,
  \[ \| J_t (\llbracket W_s^2 \rrbracket \prec I_t (w)) \|^2_{L^2} \lesssim
     \langle t \rangle^{- 3 + 4 \varepsilon} \| \llbracket W_s^2 \rrbracket
     \|^2_{\VV^{- 1 - \varepsilon}} \| I_t (w) \|^2_{H^{1 - \varepsilon}} . \]
  Now, the claim follows from interpolation and Young's inequality
  \begin{eqnarray*}
    &  & \| \llbracket W_t^2 \rrbracket \|^2_{\VV^{- 1 - \varepsilon}} \| I_t
    (w) \|^2_{H^{1 - \varepsilon}}\\
    & \lesssim & \| \llbracket W_t^2 \rrbracket \|^2_{\VV^{- 1 -
    \varepsilon}} \| I_t (w) \|^{2 - 2 \varepsilon}_{H^1} \| I_t (w) \|^{2
    \varepsilon}_{L^4}\\
    \small{(\varepsilon = 2 / 7)} \qquad & \lesssim & \| \llbracket W_t^2
    \rrbracket \|^{14}_{\VV^{- 1 - \varepsilon}} + \| I_t (w) \|^2_{H^1} + \|
    I_t (w) \|^4_{L^4}\\
    & \lesssim & \left( \int^t_0 \| w_s \|_{L^2}^2 \mathd s + \| I_t (w)
    \|^n_{W^{- 1 / 2, n + 1}} + \| \llbracket W_t^2 \rrbracket \|^{14}_{\VV^{-
    1 - \varepsilon}} \right) .
  \end{eqnarray*}
  
\end{proof}

\begin{lemma}
  \label{analyticestimate2}Let $f \in C \left( [0, \infty], \VV^{- 1 / 2 -
  \varepsilon} \right)$ and $g \in C ([0, \infty], H^1)$ such that $f_t, g_t$
  have spectral support in a ball of radius proportional to $t$. There exists
  $n \in \mathbbm{N}$ such that the following estimates hold:
  \[ \| J_t (f_t g_t^2) \|^2_{L^2} \lesssim \langle t \rangle^{- 3 / 2} \| f_t
     \|_{\VV^{- 1 / 2 - \delta}}^2 \| g_t \|^4_{L^4}, \]
  \[ \| J_t (f_t g_t^2) \|^2_{L^2} \lesssim \langle t \rangle^{- 3 / 2} \left(
     \| f_t \|^n_{\VV^{- 1 / 2 - \delta}} + \| g_t \|^2_{H^1} + \| g_t
     \|_{W^{- 1 / 2, n}}^n \right), \]
  and
  \[ \| J_t (g_t^3) \|^2_{L^2} \lesssim \langle t \rangle^{- 3 / 2} (\| g_t
     \|_{H^1}^2 + \| g_t \|_{W^{- 1 / 2, n}}^n) . \]
\end{lemma}

\begin{proof}
  By the spectral properties of $J_t$,
  \[ \| J_t (f_t g_t^2) \|^2_{L^2} \lesssim \langle t \rangle^{- 3} \| f_t
     \|^2_{L^{\infty}} \| g_t \|^4_{L^4} \lesssim \langle t \rangle^{- 3 / 2}
     \| f_t \|^2_{\VV^{- 1 / 2 - \delta}} \| g_t \|_{L^4}^4 . \]
  Applying Young's inequality with exponents $\left( \frac{n}{2}, \frac{n /
  2}{(n / 2 - 1)} \right)$ with $n$ such that $\frac{2 n}{(n / 2 - 1)}
  \leqslant 4 + \varepsilon$ where $\varepsilon$ is chosen as in Lemma
  \ref{L4estimate} we have
  \begin{eqnarray*}
    \langle t \rangle^{- 3 / 2} \| f_t \|^2_{\VV^{- 1 / 2 - \delta}} \| g_t
    \|_{L^4}^4 & \leqslant & \langle t \rangle^{- 3 / 2} \left( \| f_t
    \|^n_{\VV^{- 1 / 2 - \delta}} + \| g_t \|_{L^4}^{4 + \varepsilon}
    \right)\\
    & \leqslant & \langle t \rangle^{- 3 / 2} \left( \| f_t \|^n_{\VV^{- 1 /
    2 - \delta}} + \| g_t \|_{W^{- 1 / 2, n}}^n + \| g_t \|^2_{H^1} \right)
  \end{eqnarray*}
  Now the second estimate follows from $\tmop{chosing}$ $n$ large enough
  (depending on $\delta$) and using Besov embedding after taking $f = g$.
\end{proof}

\begin{lemma}
  The following estimates hold\label{analyticestimate3}
  \[ \langle t \rangle^{1 + \varepsilon} \| J_s (W_s I_t (w) \succ_{}
     I^{\flat}_t (u)) \|^2_{L^2} \lesssim \| I_t (w) \|^{4 +
     \varepsilon}_{L^4} + \| I^{\flat}_t (u) \|^4_{L^4} + \| W_t \|^n_{\VV^{-
     1 / 2 - \varepsilon}}, \]
  \[ \langle t \rangle^{1 + \varepsilon} \| J_s ((I_s (w))^2 \succ I^{\flat}_s
     (u)) \|^2_{L^2} \lesssim \| I_t (w) \|^{4 + \varepsilon}_{L^4} + \|
     I^{\flat}_t (u) \|^n_{\VV^{- 1 / 2 - \varepsilon}} . \]
\end{lemma}

\begin{proof}
  For the first estimate we again use the spectral properties of $W, I,$ and
  $J$ and obtain by paraproduct estimate
  \begin{eqnarray*}
    \| J_s (W_t I_t (w) \succ_{} I^{\flat}_t (u)) \|^2_{L^2} & \lesssim &
    \langle t \rangle^{- 3} \| W_t \|^2_{L^{\infty}} \| I_t (w) \|_{L^4}^2 \|
    I^{\flat}_t (u) \|^2_{L^4}\\
    & \lesssim & \langle t \rangle^{- 3} \langle t \rangle^{1 + 4
    \varepsilon} \| W_t \|^2_{\VV^{- 1 / 2 - \varepsilon}} \| I_t (w)
    \|_{L^4}^2 \| I^{\flat}_t (u) \|^2_{L^4}
  \end{eqnarray*}
  and the claim follows by Young's inequality. For the second
  \[ \| J_s ((I_s (w))^2 \succ I^{\flat}_s (u)) \|^2_{L^2} \lesssim \langle t
     \rangle^{2 - 2 \varepsilon} \| (I_s (w)) \|_{L^4}^4 \| I^{\flat}_t (u)
     \|^2_{\VV^{- 1 / 2 - \varepsilon}}, \]
  and the claim follows again by Young's inequality. 
\end{proof}

\begin{lemma}
  \label{analyticestimate4}Let $f_t \in C \left( [0, \infty], \VV^{- 1 / 2 -
  \delta} \right)$ and $g_t \in C ([0, \infty], H^1)$ such that $f_t, g_t$
  have spectral support in a ball of radius proportional to $t$. Then the
  following estimates hold
  \[ \| (J_t (f_t g_t)) \|_{L^2}^2 \lesssim \langle t \rangle^{- 1 + 2
     \delta} \| f_t \|_{\VV^{- 1 - \delta}}^2 \| g_t \|^2_{L^2} \]
  \[ \| (J_t (f_t g_t)) \|_{L^2}^2 \lesssim \langle t \rangle^{- 1 + 2 \delta}
     \left( \| f_t \|^8_{\VV^{- 1 - \delta}} + \| g_t \|^4_{H^{- 1}} + \| g_t
     \|^2_{H^1} \right) \]
\end{lemma}

\begin{proof}
  
  \[ \| (J_t (f_t g_t)) \|_{L^2}^2 \lesssim \langle t \rangle^{- 3} \| f_t
     \|_{L^{\infty}}^2 \| g_t \|^2_{L^2} \lesssim \langle t \rangle^{- 1 + 2
     \delta} \| f_t \|_{\VV^{- 1 - \delta}}^2 \| g_t \|^2_{L^2} . \]
  This proves the first estimate. For the second we continue
  \begin{eqnarray*}
    \langle t \rangle^{- 1 + 2 \delta} \| f_t \|_{\VV^{- 1 - \delta}}^2 \| g_t
    \|^2_{L^2} & \lesssim & \langle t \rangle^{- 1 + 2 \delta} \| f_t
    \|_{\VV^{- 1 - \delta}}^2 \| g_t \|_{H^1} \| g_t \|_{H^{- 1}}\\
    & \lesssim & \langle t \rangle^{- 1 + 2 \delta} \left( \| f_t \|_{\VV^{-
    1 - \delta}}^8 + \| g_t \|^4_{H^{- 1}} + \| g_t \|^2_{H^{- 1}} \right) .
  \end{eqnarray*}
  
\end{proof}

\begin{lemma}
  \label{analyticestimate5}It holds
  \[ \int^T_0 \int_{\Lambda} (J_t (\mathbbm{W}_t^2 \succ I^{\flat}_t (w)))^2
     \lesssim T^{3 \delta} \left( \sup_t \| \mathbbm{W}_t^2 \|^2_{\VV^{- 1 -
     \delta}} \right) (\sup_t \| I_t (w) \|^2_{L^2}), \]
  and
  \[ \int^T_0 \int_{\Lambda} (J_t (\mathbbm{W}_t^2 \succ I^{\flat}_t (w)))^2
     \lesssim T^{3 \delta} \left( \sup_t \| I_t (w) \|^4_{H^{- 1}} + \int^T_0
     \| w_t \|^2_{L^2} \mathd t_{} + \sup_t \| \mathbbm{W}_t^2 \|^8_{\VV^{- 1
     - \delta}} \right) . \]
  
\end{lemma}

\begin{proof}
  This follows in the same fashion as Lemma~\ref{analyticestimate4} .
\end{proof}

\appendix\section{Besov spaces and paraproducts}\label{sec:appendix-para}

In this section we will recall some well known results about Besov spaces,
embeddings, Fourier multipliers and paraproducts. The reader can find full
details and proofs
in~{\cite{bahouri_fourier_2011,gubinelli_paracontrolled_2015}}.

\

First recall the definition of Littlewood--Paley blocks. Let $\chi, \varphi$
be smooth radial functions $\mathbbm{R}^d \rightarrow \mathbbm{R}$ such that
\begin{itemize}
  \item $\tmop{supp} \chi \subseteq B (0, R)$, $\tmop{supp} \varphi \subseteq
  B (0, 2 R) \setminus B (0, R)$;
  
  \item $0 \leqslant \chi, \varphi \leqslant 1$, $\chi (\xi) + \sum_{j \geq 0}
  \varphi (2^{- j} \xi) = 1$ for any $\xi \in \mathbbm{R}^d$;
  
  \item $\tmop{supp} \varphi (2^{- j} \cdot) \cap \tmop{supp} \varphi (2^{- i}
  \cdot) = \varnothing$ if $| i - j | > 1$.
\end{itemize}
Introduce the notations $\varphi_{- 1} = \chi$, $\varphi_j = \varphi (2^{- j}
\cdot)$ for $j \geqslant 0$. For any $f \in \CS' (\Lambda)$ we define the
operators $\Delta_j f \assign \mathcal{\CF}^{- 1}_{\xi} (\varphi_j (\xi)
\hat{f} (\xi))$, $j \geqslant - 1$.

\begin{definition}
  Let $s \in \mathbbm{R}, p, q \in [1, \infty]$. For a Schwarz distribution $f
  \in \CS' (\Lambda)$ define the norm
  \[ \| f \|_{B_{p, q}^s} \assign \| (2^{j s} \| \Delta_j f \|_{L^p})_{j
     \geqslant - 1} \|_{\ell^q} . \]
  Then the space $B^s_{p, q}$ is the closure of Schwarz distributions under
  this norm. We denote $\VV^{\alpha} = B_{\infty, \infty}^{\alpha}$ the
  Besov--H{\"o}lder space and $H^{\alpha} = B^{\alpha}_{2, 2}$ the Sobolev
  spaces. \ 
\end{definition}

\begin{proposition}
  \label{besovembedding}Let $1 \leqslant p_1 \leqslant p_2 \leqslant \infty$
  and $1 \leqslant q_1 \leqslant q_2 \leqslant \infty$. Then $B^s_{p_1, q_1}$
  is continuously embedded in $B_{p_2, q_2}^{s - d \left( \frac{1}{p_1} -
  \frac{1}{p_2} \right)}$.
\end{proposition}

\begin{proposition}
  \label{compactembedding}For any $s_1, s_2 \in \mathbbm{R}$ such that $s_1 <
  s_2$, any $p, q \in [1, \infty]$ the Besov space $B^{s_1}_{p, q_{}}$ is
  compactly embedded into $B^{s_2}_{p, q}$.
\end{proposition}

\begin{definition}
  Let $f, g \in \mathcal{\CS} (\Lambda)$. We define the paraproducs
  \[ f \succ g \assign \sum_{j < i - 1} \Delta_i f \Delta_j g, \qquad
     \text{and} \qquad f \prec g \assign \sum_{j > i + 1} \Delta_i f \Delta_j
     g = g \succ f. \]
  Moreover we introduce the resonant product
  \[ f \circ g \assign \sum_{| i - j | \leqslant 1} \Delta_i f \Delta_j g. \]
  Then $fg = f \prec g + f \circ g + f \succ g$.
\end{definition}

\begin{proposition}
  \label{paraproductestimate}Let $\alpha < 0, \beta \in \mathbbm{R}$. For $f,
  g \in \CS (\Lambda)$ we have the estimates
  \[ \begin{array}{lllllll}
       \| f \succ g \|_{H^{\beta - \delta}} & \lesssim & \| f \|_{\VV^{\beta}}
       \| g \|_{L^2}, &  & \| f \succ g \|_{\VV^{\beta}} & \lesssim & \| f
       \|_{\VV^{\beta}} \| g \|_{L^{\infty}},\\
       \| f \succ g \|_{H^{\beta - \alpha}} & \lesssim & \| f \|_{\VV^{\beta}}
       \| g \|_{H^{\alpha}}, & \qquad & \| f \succ g \|_{\VV^{\beta}} &
       \lesssim & \| f \|_{\VV^{\beta}} \| g \|_{\VV^{\alpha}} .
     \end{array} \]

  Let $\alpha, \beta \in \mathbbm{R}$ such that $\alpha + \beta > 0$. Then
  \[ \begin{array}{lllllll}
       \| f \circ g \|_{H^{\alpha + \beta}} & \lesssim & \| f \|_{\VV^{\beta}}
       \| g \|_{H^{\alpha}}, & \qquad & \| f \circ g \|_{\VV^{\alpha + \beta}}
       & \lesssim & \| f \|_{\VV^{\beta}} \| g \|_{\VV^{\alpha}} .
     \end{array} \]
  By density the paraproduct and resonant product also extend to bilinear
  operators on the respective spaces. 
\end{proposition}

\begin{proposition}
  \label{commutatorestimate}Let $\alpha \in (0, 1)$ $\beta, \gamma \in
  \mathbbm{R}$ such that $\beta + \gamma < 0$, $\alpha + \beta + \gamma > 0$.
  Then for $f, g, h \in \CS$, and for any $\delta > 0$,
  \[ \| (f \succ g) \circ h - g (f \circ h) \|_{H^{\alpha + \beta + \gamma -
     \delta}} \lesssim \| f \|_{\VV^{\gamma}} \| h \|_{\VV^{\beta}} \| g
     \|_{H^{\alpha}}, \]
  \[ \| (f \succ g) \circ h - g (f \circ h) \|_{\VV^{\alpha + \beta + \gamma}}
     \lesssim \| f \|_{\VV^{\gamma}} \| h \|_{\VV^{\beta}} \| g
     \|_{\VV^{\alpha}} . \]
\end{proposition}

\begin{proposition}
  \label{adjointparaproduct}Assume $f \in \VV^{\alpha}, g \in H^{\beta}, h \in
  H^{\gamma}$ and $\alpha + \beta + \gamma = 0$. Then
  \[ \int_{\mathbbm{T}^d} [(f \succ g) h - (f \circ h) g] \lesssim \| f
     \|_{\VV^{\alpha}} \| g \|_{H^{\beta}} \| h \|_{H^{\gamma}} . \]
\end{proposition}

\begin{remark}
  Proposition~\ref{adjointparaproduct} is not proven in the above references
  but is quite easy and the reader can fill out a proof.
\end{remark}

\begin{definition}
  A smooth function $\eta : \mathbbm{R}^d \rightarrow \mathbbm{R}$ is said to
  be an $S^m$-multiplier if for every multi-index $\alpha$ there exists a
  constant $C_{\alpha}$ such that
  \begin{equation}
    \quad \left| \frac{\partial^{\alpha}}{\partial \xi^{\alpha}} f (\xi)
    \right| \lesssim_{\alpha} (1 + | \xi |)^{m - | \alpha |}, \qquad \qquad
    \qquad \forall \xi \in \mathbbm{R}^d . \label{symbolinq}
  \end{equation}
  We say that a family $(\eta_t)_{t \geqslant 0}$ is a uniform
  $S^m$-multiplier if (\ref{symbolinq}) is satisfied for every $\eta_t$ with
  $C_{\alpha}$ independent of $t \geqslant 0$. 
\end{definition}

\begin{proposition}
  \label{multiplierestimate}Let $\eta$ be an $S^m$-multiplier, $s \in
  \mathbbm{R}$, $p, q \in [1, \infty]$, and $f \in B_{p, q}^s
  (\mathbbm{T}^d)$, then
  \[ \| \eta (\mathD) f \|_{B_{p, q}^{s - m}} \lesssim \| f \|_{B_{p, q}^s} .
  \]
  Furthermore the constant depends only on $s, p, q, d$ and the constants
  $C_{\alpha}$ in (\ref{symbolinq}).
\end{proposition}

\begin{proposition}
  \label{paraproductleibniz}Assume $m \leqslant 0$, $\alpha \in (0, 1), \beta
  \in \mathbbm{R}$. Let $\eta$ be an $S^m$-multiplier, $f \in \VV^{\beta}$, $g
  \in H^{\alpha}$. Then for any $\delta > 0$.
  \[ \| \eta (\mathD) (f \succ g) - (\eta (\mathD) f \succ g) \|_{H^{\alpha +
     \beta - m - \delta}} \lesssim \| f \|_{\VV^{\beta}} \| g \|_{H^{\alpha}}
     . \]
  Again the constant depends only on $\alpha, \beta, \delta$ and the constants
  in~(\ref{symbolinq}).
\end{proposition}

\begin{proposition}
  \label{besovembedding2}Let $\delta > 0$.We have for any $q_1, q_2 \in [1,
  \infty], q_1 < q_2$
  \[ \| f \|_{B_{p, q_2}^s} \leqslant \| f \|_{B_{p, q_1}^s} \leqslant \| f
     \|_{B_{p, \infty}^{s + \delta}} . \]
  Furthermore, if we denote by $W^{s, p}, s \in \mathbbm{R}, p \in [1,
  \infty]$ the fractional Sobolev spaces defined by the norm $\| f \|_{W^{s,
  q}} \assign \| \langle \mathD \rangle^s f \|_{L^q}$, then, for any $q \in
  [1, \infty]$,
  \[ \| f \|_{B_{p, q}^s} \leqslant \| f \|_{W^{s + \delta, p}} \leqslant \| f
     \|_{B_{p, \infty}^{s + 2 \delta}} . \]
\end{proposition}

\

\


\begin{thebibliography}{10}
  \bibitem[1]{albeverio_remark_2008}S.~Albeverio  and  S.~Liang. {\newblock}A
  remark on the nonequivalence of the time-zero $\Phi^4_3$-measure with the
  free field measure. {\newblock}\tmtextit{Markov Processes and Related
  Fields}, 14(1):159--164, 2008.{\newblock}
  
  \bibitem[2]{albeverio_invariant_2017}S.~Albeverio  and  S.~Kusuoka.
  {\newblock}The invariant measure and the flow associated to the
  $\Phi^4_3$-quantum field model. {\newblock}\tmtextit{Annali della Scuola
  Normale di Pisa - Classe di Scienze}, 2018.
  {\newblock}\href{Https://doi.org/10.2422/2036-2145.201809_008}{}10.2422/2036-2145.201809\_008.{\newblock}
  
  \bibitem[3]{bahouri_fourier_2011}H.~Bahouri, J.-Y.~Chemin, and  R.~Danchin.
  {\newblock}\tmtextit{Fourier Analysis and Nonlinear Partial Differential
  Equations}. {\newblock}Springer, jan 2011.{\newblock}
  
  \bibitem[4]{barashkov_gubinelli_variational}N.~Barashkov  and  M.~Gubinelli.
  {\newblock}A variational method for $\Phi^4_3$.
  {\newblock}\tmtextit{\href{Http://arxiv.org/abs/1805.10814}{}arXiv:1805.10814},
  2018.{\newblock}
  
  \bibitem[5]{benfatto_ultraviolet_1980}G.~Benfatto, M.~Cassandro,
  G.~Gallavotti, F.~Nicol{\'o}, E.~Olivieri, E.~Presutti, and 
  E.~Scacciatelli. {\newblock}Ultraviolet stability in Euclidean scalar field
  theories. {\newblock}\tmtextit{Communications in Mathematical Physics},
  71(2):95--130, jun 1980.
  {\newblock}\href{Https://doi.org/10.1007/BF01197916}{}10.1007/BF01197916.{\newblock}
  
  \bibitem[6]{brydges_new_1983}D.~C.~Brydges, J.~Fr{\"o}hlich, and 
  A.~D.~Sokal. {\newblock}A new proof of the existence and nontriviality of
  the continuum $\phi^4_2$ and $\phi^4_3$ quantum field theories.
  {\newblock}\tmtextit{Communications in Mathematical Physics},
  91(2):141--186, 1983.{\newblock}
  
  \bibitem[7]{catellier_paracontrolled_2013}R.~Catellier  and  K.~Chouk.
  {\newblock}Paracontrolled distributions and the 3-dimensional stochastic
  quantization equation. {\newblock}\tmtextit{The Annals of Probability},
  46(5):2621--2679, 2018.
  {\newblock}\href{Https://doi.org/10.1214/17-AOP1235}{}10.1214/17-AOP1235.{\newblock}
  
  \bibitem[8]{feldman_lambda_1974}J.~Feldman. {\newblock}The $\lambda
  \varphi^4_3$ field theory in a finite volume.
  {\newblock}\tmtextit{Communications in Mathematical Physics}, 37:93--120,
  1974.{\newblock}
  
  \bibitem[9]{feldman_wightman_1976}J.~S.~Feldman  and  K.~Osterwalder.
  {\newblock}The Wightman axioms and the mass gap for weakly coupled
  $\Phi^4_3$ quantum field theories. {\newblock}\tmtextit{Annals of Physics},
  97(1):80--135, 1976.{\newblock}
  
  \bibitem[10]{glimm_positivity_1973}J.~Glimm  and  A.~Jaffe.
  {\newblock}Positivity of the $\phi^4_3$ Hamiltonian.
  {\newblock}\tmtextit{Fortschritte der Physik. Progress of Physics},
  21:327--376, 1973.
  {\newblock}\href{Https://mathscinet.ams.org/mathscinet-getitem?mr=0408581}{}MR0408581.{\newblock}
  
  \bibitem[11]{glimm_quantum_1987}J.~Glimm  and  A.~Jaffe.
  {\newblock}\tmtextit{Quantum Physics: A Functional Integral Point of View}.
  {\newblock}Springer-Verlag, New York, 2  edition, 1987.{\newblock}
  
  \bibitem[12]{gubinelli_pde_2018}M.~Gubinelli  and  M.~Hofmanov{\'a}.
  {\newblock}A PDE construction of the Euclidean $\Phi^4_3$ quantum field
  theory.
  {\newblock}\tmtextit{\href{Http://arxiv.org/abs/1810.01700}{}arXiv:1810.01700},
  2018.{\newblock}
  
  \bibitem[13]{gubinelli_global_2019}M.~Gubinelli  and  M.~Hofmanov{\'a}.
  {\newblock}Global Solutions to Elliptic and Parabolic $\Phi^4$ Models in
  Euclidean Space. {\newblock}\tmtextit{Communications in Mathematical
  Physics}, 368(3):1201--1266, 2019.{\newblock}
  
  \bibitem[14]{gubinelli_paracontrolled_2015}M.~Gubinelli, P.~Imkeller, and 
  N.~Perkowski. {\newblock}Paracontrolled distributions and singular PDEs.
  {\newblock}\tmtextit{Forum of Mathematics. Pi}, 3:0, 2015.
  {\newblock}\href{Https://doi.org/10.1017/fmp.2015.2}{}10.1017/fmp.2015.2.{\newblock}
  
  \bibitem[15]{gulisashvili_exact_1996}A.~Gulisashvili  and  M.~A.~Kon.
  {\newblock}Exact Smoothing Properties of Schr{\"o}dinger Semigroups.
  {\newblock}\tmtextit{American Journal of Mathematics}, 118(6):1215--1248,
  1996. {\newblock}\href{Http://www.jstor.org/stable/25098514}{}JSTOR
  25098514.{\newblock}
  
  \bibitem[16]{hairer_theory_2014}M.~Hairer. {\newblock}A theory of regularity
  structures. {\newblock}\tmtextit{Inventiones mathematicae}, 198(2):269--504,
  2014.
  {\newblock}\href{Https://doi.org/10.1007/s00222-014-0505-4}{}10.1007/s00222-014-0505-4.{\newblock}
  
  \bibitem[17]{jona_lasinio_stochastic_1985}G.~Jona-Lasinio  and 
  P.~K.~Mitter. {\newblock}On the stochastic quantization of field theory.
  {\newblock}\tmtextit{Communications in Mathematical Physics (1965-1997)},
  101(3):409--436, 1985.{\newblock}
  
  \bibitem[18]{kupiainen_renormalization_2016}A.~Kupiainen.
  {\newblock}Renormalization Group and Stochastic PDEs.
  {\newblock}\tmtextit{Annales Henri Poincar{\'e}}, 17(3):497--535, 2016.
  {\newblock}\href{Https://doi.org/10.1007/s00023-015-0408-y}{}10.1007/s00023-015-0408-y.{\newblock}
  
  \bibitem[19]{magnen_infinite_1976}J.~Magnen  and  R.~S{\'e}n{\'e}or.
  {\newblock}The infinite volume limit of the $\phi^4_3$ model.
  {\newblock}\tmtextit{Ann. Inst. H. Poincar{\'e} Sect. A (N.S.)},
  24(2):95--159, 1976.
  {\newblock}\href{Https://mathscinet.ams.org/mathscinet-getitem?mr=0406217}{}MR0406217.{\newblock}
  
  \bibitem[20]{moinat_space_time_2018}A.~Moinat  and  H.~Weber.
  {\newblock}Space-time localisation for the dynamic $\Phi^4_3$ model.
  {\newblock}\tmtextit{ArXiv:1811.05764}, nov 2018. {\newblock}ArXiv:
  1811.05764.{\newblock}
  
  \bibitem[21]{MWcomedown}J.-C.~Mourrat  and  H.~Weber. {\newblock}The dynamic
  $\Phi^4_3$ model comes down from infinity. {\newblock}\tmtextit{Comm. Math.
  Phys.}, 356(3):673--753, 2017.{\newblock}
  
  \bibitem[22]{parisi_perturbation_1981}G.~Parisi  and  Y.~S.~Wu.
  {\newblock}Perturbation theory without gauge fixing.
  {\newblock}\tmtextit{Scientia Sinica. Zhongguo Kexue}, 24(4):483--496,
  1981.{\newblock}
  
  \bibitem[23]{park_lambda_1975}Y.~M.~Park. {\newblock}The $\lambda
  \varphi^4_3$ Euclidean quantum field theory in a periodic box.
  {\newblock}\tmtextit{Journal of Mathematical Physics}, 16(11):2183--2188,
  1975.
  {\newblock}\href{Https://doi.org/10.1063/1.522464}{}10.1063/1.522464.{\newblock}
\end{thebibliography}
\end{document}